\documentclass[leqno,twoside]{article}
\usepackage[all]{xy}
\usepackage{amsmath}
\usepackage{amssymb}
\usepackage{amsthm}
\usepackage{footmisc}

\newtheorem{theorem}{Theorem}[section]
\newtheorem{lemma}[theorem]{Lemma}
\newtheorem{corollary}[theorem]{Corollary}

\newtheorem{proposition}[theorem]{Proposition}
\theoremstyle{definition}  
\newtheorem{definition} [theorem] {Definition} 

\newtheorem{example} [theorem] {Example}
\newtheorem{remark} [theorem] {Remark}

\newcommand{\C}{{\mathbb{C}}}

\newcommand{\PP}{{\mathbb{P}}}
\newcommand{\R}{{\mathbb{R}}}
\newcommand{\Z}{{\mathbb{Z}}}
\newcommand{\bmu}{{\mbox{\boldmath{$\mu$}}}}
\newcommand{\sbmu}{{\mbox{\boldmath{\scriptsize$\mu$}}}}
\newcommand{\calP}{{\mathcal{P}}}
\newcommand{\calCP}{{\mathcal{CP}}}
\newcommand{\sfk}{{\mathsf{k}}}
\def\Grass{\mathrm{Grass}}
\def\cod{\mathrm{codim}}
\def\Mat{\mathrm{Mat}}

\DeclareMathSymbol{\notdiv}{\mathbin}{AMSb}{"2D}

\numberwithin{equation}{section}

\begin{document}

\title{Strata of rational space curves}

\author{David A.\ Cox\footnote{corresponding author}\\
Department of Mathematics \& Statistics\\
Amherst College, Amherst, MA 01002, USA\\
dacox@amherst.edu\\
{}\\
Anthony A.\ Iarrobino\\ 
Department of Mathematics\\
Northeastern University, Boston MA 02115, USA\\
a.iarrobino@neu.edu}

\date{}

\maketitle

\begin{abstract}
The $\bmu$-invariant $\bmu = (\mu_1,\mu_2,\mu_3)$ of a rational space
curve gives important information about the curve.  In this paper, we
describe the structure of all parameterizations that have the same
$\bmu$-type, what we call a \emph{$\bmu$-stratum}, and as well the
closure of strata.  Many of our results are based on papers by the
second author that appeared in the commutative algebra literature.  We
also present new results, including an
explicit formula for the codimension of the locus of non-proper
parametrizations within each $\bmu$-stratum and a decomposition of the
smallest $\bmu$-stratum based on which two-dimensional rational normal
scroll the curve lies on.
\end{abstract}

\let\thefootnote\relax{}\footnote{\hspace*{-4pt}2010 \emph{Mathematics
    Subject Classification.} 14Q05; 13D40, 65D17}

\let\thefootnote\relax{}\footnote{\hspace*{-4pt}\emph{Key words and phrases.} rational space curve, Hilbert function, parametrization, non-proper maps, $\sbmu$-strata, rational normal scroll}

\vspace*{-20pt}


\section{Introduction}
\label{intro}

A rational curve of degree $n$ in projective $3$-space is parametrized by 
\begin{equation}
\label{spaceparam}
F(s,t) = (a_0(s,t),a_1(s,t), a_2(s,t),a_3(s,t))
\end{equation}
where $a_0,a_1,a_2,a_3$ are relatively prime homogeneous polynomials
of degree $n$.  If $F$ is generically one-to-one and $a_0,a_1,a_2,a_3$
are linearly independent, then the image curve $C$ has degee $n$ and
does not lie in a plane, i.e., $C$ is a genuine \emph{space curve}.  

For a parametrized planar curve of degree $n$, the 1998 paper
\cite{CSC} introduced the idea of a \emph{$\mu$-basis}.  Since then,
$\mu$-bases have proved useful in the study of the singularities of
rational plane curves, as evidenced by the papers \cite{FWL,SCG} in
the geometric modeling literature and \cite{CKPU} in the commutative
algebra literature. 

The groundwork for the space curve case appears in
\cite[Sec.\ 5]{CSC}, and the connection with singularities has been
studied in several papers, including \cite{JWG,SC,SJG,WJG}.  For
references to the (fairly extensive) algebraic geometry literature on
rational space curves, we direct the reader to the bibliography of
\cite{I3}. 

An idea introduced in \cite{CSC} was to study the $\mu$-stratum
consisting of \emph{all} pa\-ram\-e\-trizations of plane rational
curves with given degree and $\mu$-type.  In this paper, we will
extend this idea to rational space curves and more generally to define
the $\bmu$-strata of rational curves in projective $d$-space, based on
the papers \cite{I1,I2} of the second author.  A terse version of the
results presented in Sections~\ref{higher} and \ref{smallest}, written
for commutative algebraists, can be found in \cite{I3}. The results
concerning properness in Section \ref{nonproper} are new to this paper
and this topic is not mentioned in \cite{I1,I2}.  Results concerning
the decomposition of the smallest $\bmu$-stratum are geometric
consequences of \cite{I2}, but were not developed there.

We will review the planar case in Section~\ref{plane} and then discuss
$\bmu$-strata in higher dimensions in Section~\ref{higher}.  We will
give examples to illustrate the unexpected behaviors that can arise.
Section~\ref{further} will study non-proper parametrizations and
explain how parametrizations in the smallest $\bmu$-stratum relate to
two-dimensional rational normal scrolls in $\PP^d$.  Proofs will be
given in Appendix~\ref{proofs}.


\section{Planar Rational Curves}
\label{plane}

For the rest of the paper, we will work over an arbitrary infinite
field $\sfk$, which in practice is usually $\sfk = \R$ or $\C$.  Set
$R = \sfk[s,t]$ and let $R_n$ be the subspace consisting of
homogeneous polynomials of degree $n$.

A rational curve in the projective plane is parametrized by
\begin{equation}
\label{planeparam}
F(s,t) = (a_0(s,t),a_1(s,t), a_2(s,t)),
\end{equation}
where $a_0,a_1,a_2 \in R_n$.  In this section, we will assume that
$a_0,a_1,a_2$ are relatively prime and linearly independent and that
$F$ is generically one-to-one.   

A \emph{moving line} of degree $m$ is a polynomial $A_0(s,t) x +
A_1(s,t) y + A_2(s,t) z$ with $A_0,A_1,A_2 \in R_m$.  It
\emph{follows} the parametrization if
\begin{equation} 
\label{planesyz}
A_0a_0+A_1a_1 + A_2a_2 = 0 \ \text{in}\  R.
\end{equation}
A \emph{$\mu$-basis} for $F$ consists of a pair of moving lines $p,q$
that follow the parametrization and have the property that \emph{any}
moving line that follows the parametrization is a linear combination
(with polynomial coefficients) of $p$ and $q$.  Assuming $\deg(p) \le
\deg(q)$, one sets $\mu = \deg(p)$, so that $\deg(q) = n-\mu$ since it
is known that $\deg(p) + \deg(q) = n$.  In this situation, we say that
$F$ has \emph{type} $\mu$ (this is the terminology used in \cite{SJG}).
Thus the $\mu$-type is the minimum degree of a moving line that
follows $F$.  Note that
\[
1 \le \mu \le \lfloor n/2\rfloor
\]
since $a_0,a_1,a_2$ are linearly independent and $\mu \le n - \mu$.

To connect this with algebraic geometry and commutative algebra, we
introduce the ideal $I = \langle a_0,a_1,a_2\rangle \subseteq R$.
Then, as explained in \cite{CSC}, the Hilbert Syzygy Theorem gives an
exact sequence
\begin{equation}
\label{planeexact}
0 \longrightarrow R(-n-\mu) \oplus R(-2n+\mu) \stackrel{A}
{\longrightarrow} R(-n)^3 \stackrel{B}{\longrightarrow} I
\longrightarrow 0, 
\end{equation}
where $B$ is given by $(a_0,a_1,a_2)$ and $A$ is the $3\times2$ matrix
whose columns are the coefficients of $p$ and $q$, and $BA=0$.  The notation
$R(-n-\mu), R(-2n+\mu), R(-n)$ reflects the degree shifts needed to
make the maps in \eqref{planeexact} preserve degrees.  We note that
$\mu$-bases and $\mu$-types appeared in the algebraic geometry
literature as early as 1986 (see \cite{Asc1,Asc2}).

When we think of $p$ and $q$ as columns of the matrix $A$, then
the Hilbert-Burch Theorem asserts that the cross product $p \times q$
is the parametrization $(a_0,a_1,a_2)$, up to multiplication by a
nonzero constant.  This feature makes it easy to create
parametrizations with given $\mu$-type: just choose generic $p$ and
$q$ of respective degrees $\mu$ and $n-\mu$ and take their cross product.  

To study all parametrizations with the same $\mu$-type, we begin with
the subset $\calP_n \subseteq R_n^3$ consisting of all relatively
prime linearly independent triples $(a_0,a_1,a_2)$ for which the
parametrization is generically one-to-one.  Then, for $1 \le \mu \le
\lfloor n/2 \rfloor$, we have the \emph{$\mu$-stratum}
\[
\calP^{\mu}_n = \{(a_0,a_1,a_2) \in \calP_n \mid
(a_0,a_1,a_2)\ \text{has type }\mu\}. 
\]  
In \cite{CSC}, it was proved that $\calP^{\mu}_n$ is irreducible
of dimension 
\begin{equation}
\label{planedim}
\dim(\calP_n^\mu) = \begin{cases} 3(n+1) ,& \mbox{if}\ \mu = \lfloor
  n/2\rfloor,\\  
	2n+2\mu +4, & \mbox{if}\ \mu < \lfloor n/2\rfloor.\end{cases}
\end{equation}

The $\mu$-stratum $\calP^{\mu}_n$ is not closed in $\calP_n$.  Let
$\overline{\calP}^\mu_n$ denote its Zariski closure in $\calP_n$.
In \cite{CSC}, it was conjectured that
\begin{equation}
\label{planestrat}
\overline{\calP}_n^\mu = \calP^1_n \cup \cdots \cup \calP^{\mu}_n. 
\end{equation}
This was proved in 2004 by D'Andrea \cite{D}.  This result also is a
consequence of the 1977 memoir \cite{I1} or the 2004 article
\cite{I2} by the second author, though these are written from a
very different viewpoint.  It was eight years after \cite{D}
appeared before a connection was made between these papers.

Here is the intuition behind \eqref{planedim} and \eqref{planestrat}:
\begin{itemize}
\item \eqref{planedim} says that the smaller the $\mu$, the more
  special the parame\-trization. 
\item \eqref{planestrat} says that if we are moving around in
  $\calP^{\mu}_n$ and reach its boundary, then with high
  probability, we hit a point of $\calP^{\mu-1}_n$, i.e., $\mu$ drops
  by one unless we are really unlucky.  
\end{itemize}
We will see in the next section that the  results in \cite{I1,I2} also
apply to parametrizations in projective $d$-space, though the analogs
of the above two bullets become more sophisticated as $d$ increases.   

\section{Rational Curves in Projective Space}
\label{higher}

Curves in the plane and $3$-dimensional space are the most important
to geometric modeling.  Since the results of \cite{I1,I2} apply to
curves in projective $d$-space for all $d \ge 1$, we will work in this
greater generality.  One way to think of our approach is that it gives
a unified treatment of rational plane and space curves, as well as
rational curves in higher dimensions.

We therefore start with $d+1$ homogeneous polynomials $a_0,\dots,a_d
\in R_n$.  Then the function
\[
F(s,t) = (a_0(s,t),\dots,a_d(s,t))
\]
parametrizes a curve in projective $d$-space, generalizing
\eqref{spaceparam} and \eqref{planeparam}.  We will assume that
$a_0,\dots,a_d$ are linearly independent, which implies that the image
curve is not contained in a hyperplane.  Note also that the span
$\mathrm{Span}(a_0,\dots,a_d)$ is a $(d+1)$-dimensional subspace of
$R_n$.  This is important, since the results of \cite{I1,I2} for $R =
\sfk[s,t]$ are stated in terms of subspaces of $R_n$ of a given
dimension, equal to $d+1$ in our situation.  We will say more about
this in Appendix~\ref{proofs}.

In Section~\ref{plane}, we made two assumptions beyond linear
independence:
\begin{itemize}
\item The parametrization is proper, i.e., generically one-to-one, and
\item The polynomials in the parametrization are relatively prime.
\end{itemize}
In this section, we will dispense with the first assumption, so that
we allow non-proper param\-etrizations.  We will see in
Section~\ref{nonproper} that this is harmless.  As for the second
assumption, we will give two versions of our main results, one that
assumes relatively prime, and one that does not.

\subsection{$\bmu$-Bases and $\bmu$-Types}
Before stating our results, we need to define $\bmu$-types and
$\bmu$-bases.  

\begin{proposition}
\label{muprop}
Let $a_0,\dots,a_d \in R_n$ be linearly independent and let 
\[
I = \langle a_0,\dots,a_d\rangle \subseteq R
\]  
be the the ideal generated by $a_0,\ldots, a_d$. Then there exist
integers $\mu_1,\dots,\mu_d \ge 1$ and an exact sequence
\begin{equation}
\label{genexact}
0 \longrightarrow {\textstyle\bigoplus_{i=1}^d} R(-n-\mu_i)
\stackrel{A} {\longrightarrow} R(-n)^{d+1}
\stackrel{B}{\longrightarrow} I \longrightarrow 0.
\end{equation}
Furthermore, if we set $h = \gcd(a_0,\dots,a_d)$,
then $\mu_1+\cdots+\mu_d = n-c$, $c = \deg(h)$, and $B =
(a_0,\dots,a_d)$ consists of $h$ times the maximal minors of $A$
(up to sign).
\end{proposition}

\begin{proof}
Let $h = \gcd(a_0,\dots,a_d)$ and set $b_i = a_i/h$.  Then
$\gcd(b_0,\dots,b_d) = 1$ and $\deg(b_i) = n-c$ since $c = \deg(h)$.
The results of \cite[Sec.\ 5]{CSC} apply to $b_0,\dots,b_d$, so that
we have an exact sequence
\[
0 \longrightarrow {\textstyle\bigoplus_{i=1}^d} R(-n-c-\mu_i)
\stackrel{A} {\longrightarrow} R(-n-c)^{d+1}
\stackrel{\widetilde{B}}{\longrightarrow} \langle b_0,\dots,b_d\rangle
\longrightarrow 0, 
\]
where $\widetilde{B} = (b_0,\dots,b_d)$ and $A$ is a $(d+1)\times d$
matrix whose columns form a basis of the module of moving hyperplanes
$A_0x_0 + \cdots + A_dx_d$ that follow the parametrization given by
$(b_0,\dots,b_d)$.  Since $\deg(b_i) = n-c$, \cite{CSC} also
implies that $\mu_1+\cdots+\mu_d = n-c$.  Since $a_i = h b_i$ and $h
\ne 0$, we have 
\[
\sum_{i=1}^d A_i a_i = 0 \iff \sum_{i=1}^d A_i b_i = 0.
\]
Setting $B = h\widetilde{B}$ gives the exact sequence
\eqref{genexact}.  The $b_i$ are (up to sign) the maximal minors of
$A$ by \cite{CSC}, and the final assertion of the proposition follows.
\end{proof}

If we require $\mu_1 \le \cdots \le \mu_d$, then the $d$-tuple
\begin{equation}
\label{muinvt}
\bmu = (\mu_1,\dots,\mu_d)
\end{equation}
is uniquely determined by $(a_0,\dots,a_d)$.  We call \eqref{muinvt}
the \emph{$\bmu$-type} of $(a_0,\dots,a_d)$, and the columns of the
matrix $A$ in \eqref{genexact} form a \emph{$\bmu$-basis}.  Note
that when $d = 2$ and $\gcd(a_0,a_1,a_2) =1$, Proposition~\ref{muprop}
tells us that the $\bmu$-type can be written
\[
\bmu = (\mu,n-\mu).
\]
Hence we recover the $\mu$-type in the planar case when the
polynomials are relatively prime.  

\begin{remark} We use ``$\bmu$'' in two ways in this paper.  When
  followed by a hyphen, as in $\bmu$-type, $\bmu$-basis or
  $\bmu$-stratum, the $\bmu$ is part of the notation and has no
  specific value.  But when used by itself, $\bmu$ denotes a vector of
  integers, such as $\bmu = (1,2,3)$.
\end{remark}

\subsection{The Relatively Prime Case}\label{relativelyprime}
To state our first result, let 
\[
\calP_{n,d} \subseteq R_n^{d+1}
\]
consist of all $(a_0,\dots,a_d) \in R_n^{d+1}$ such that
$a_0,\dots,a_d$ are linearly independent and relatively prime.  Linear
independence means that the image of the parametrization does not lie
in any hyperplane of $\PP^d$.  We will always assume that $n \ge d$
(otherwise $\calP_{n,d}$ is empty) and $d \ge 2$ (curves in $\PP^1$
are not interesting).  One can show that $\calP_{n,d}$ is a nonempty
Zariski open subset of $R_n^{d+1}$, i.e., the complement of a proper
closed subvariety of $R_n^{d+1}$.

Given integers with $1 \le \mu_1 \le \cdots \le \mu_d$, we set $\bmu =
(\mu_1,\dots, \mu_d)$ and $|\bmu| = \mu_1 + \cdots + \mu_d$.  We call
$\bmu$ a \emph{$d$-part partition}.  (Many references write partitions
in descending order, e.g., $7 = 4 + 2 + 1$.  We use ascending order
since this is how $\bmu$-bases are written in the geometric modeling
literature.)

If $\bmu$ is a $d$-part partition of $n$ (so $n = |\bmu|$), then we
define the \emph{$\bmu$-stratum}
\begin{equation}
\label{stratadef1}
\calP^\sbmu_{n,d} = \{(a_0,\dots,a_d) \in \calP_{n,d} \mid
(a_0,\dots,a_d)\ \text{has type }\bmu\}. 
\end{equation}
This set has the following structure.

\begin{theorem}\label{dimPndthm}
\label{relprimedim}
Assume that $\bmu$ is a $d$-part partition of $n$.  Then
$\calP_{n,d}^\sbmu$ is a Zariski open subset of a subvariety of
$R_n^{d+1} \simeq \sfk^{(d+1)(n+1)}$.  Furthermore,
$\calP_{n,d}^\sbmu$ is irreducible of dimension
\[
\dim(\calP_{n,d}^\sbmu) = (d+1)(n+1)- \sum_{i > j} \max(0,\mu_i -
  \mu_j -1).
\]
\end{theorem}

The proofs of all theorems in this section will be given in
Appendix~\ref{proofs}.  

\begin{example}
In the plane case, we have $d = 2$ and $\bmu = (\mu,n-\mu)$.  Here,
$\calP_{n,d}^\sbmu = \calP_n^\mu$ as defined in Section~\ref{plane},
and then Theorem~\ref{relprimedim} implies that
\begin{align*}
\dim(\calP_n^\mu) &= 3(n+1) - \max(0, (n-\mu) - \mu
-1)\\ &= \begin{cases} 3(n+1), & \mbox{if}\ n-\mu = \mu, \text{ i.e.,
  } \mu = 
  \lfloor n/2\rfloor,\\ 2n+2\mu+4, & \mbox{if}\ n-\mu > \mu, \text{
    i.e., } \mu < 
  \lfloor n/2\rfloor, \end{cases}
\end{align*}
in agreement with \eqref{planedim}.  
\end{example}

\begin{example}
A space curve case studied in \cite{JWG} is $\bmu = (1,1,n-2)$.
Assuming $n \ge 4$, one computes that
\[
\dim(\calP_{n,3}^\sbmu) = 4(n+1) - 2\max(0, (n-2)-1 -1) = 2n+12.
\]
We will soon see that this is the smallest $\bmu$-stratum of $\calP_{n,3}$.
\end{example}

In general, $\calP_{n,d}^\sbmu$ is not a subvariety of $\calP_{n,d}$.
We let $\overline{\calP}_{n,d}^\sbmu$ denote its Zariski closure in
$\calP_{n,d}$, i.e., the smallest subvariety of $\calP_{n,d}$
containing ${\calP}_{n,d}^\sbmu$.  Theorem~\ref{relprimedim}
tells us that $\calP_{n,d}^\sbmu$ is Zariski open in
$\overline{\calP}_{n,d}^\sbmu$.  The theorem also implies that
$\overline{\calP}_{n,d}^\sbmu$ is irreducible with
\[
\dim(\overline{\calP}_{n,d}^\sbmu) = \dim(\calP_{n,d}^\sbmu).
\]

The expectation is that the complement $\overline{\calP}_{n,d}^\sbmu
\setminus \calP_{n,d}^\sbmu$ should consist of ``smaller''
$\bmu$-strata.  We compare different $\bmu$-types as follows.

\begin{definition}\label{orderdef}
Given $d$-part partitions $\bmu$ and
$\bmu'$, we define $\bmu \le \bmu'$ provided
\[
\mu_1 \le \mu_1',\ \mu_1+\mu_2 \le
\mu_1'+\mu_2',\ \dots,\ \mu_1+\cdots+ \mu_{d} \le
\mu_1'+\cdots+\mu_{d}'.
\]
Note that $\bmu \le \bmu'$ implies $|\bmu| \le |\bmu'|$.  
\end{definition}

We can now describe the Zariski closure of $\calP_{n,d}^\sbmu$ in
$\calP_{n,d}$.  Recall from the discussion leading up to
\eqref{stratadef1} that all $\bmu$-strata occuring in $\calP_{n,d}$
satisfy $|\bmu| = n$.

\begin{theorem}
\label{relprimeclosure}
The Zariski closure of $\calP_{n,d}^\sbmu$ in $\calP_{n,d}$ is
$\overline{\calP}_{n,d}^\sbmu = \bigcup_{\sbmu' \le \sbmu}
\calP_{n,d}^{\sbmu'}$.
\end{theorem}

Since $\bmu$-strata are disjoint, this theorem implies that
$\overline{\calP}_{n,d}^\sbmu \setminus \calP_{n,d}^{\sbmu} =
\bigcup_{\sbmu' < \sbmu} \calP_{n,d}^{\sbmu'}$, confirming our
intuition that $\overline{\calP}_{n,d}^\sbmu \setminus
\calP_{n,d}^{\sbmu}$ is a union of smaller strata.

\begin{example}
Since $(\mu,n-\mu) \le (\mu',n-\mu')$ if and only if $\mu \le \mu'$,
we see that Theorem~\ref{relprimeclosure} reduces to
\eqref{planestrat} when $d = 2$.
\end{example}

\begin{example}
One easily checks that $(1,1,n-2) \le \bmu$ for all $3$-part
partitions $\bmu$ of $n$.  This and Theorem~\ref{relprimeclosure}
justify our earlier claim that $\calP_{n,3}^{(1,1,n-2)}$ is the
smallest $\bmu$-stratum of $\calP_{n,3}$.
Proposition~\ref{bigsmallprop} will describe the smallest
$\bmu$-stratum of $\calP_{n,d}$.

\end{example}

\begin{example}
\label{ex63}
Sextic curves in dimension $3$ have been studied in \cite{JWG}.  Here,
the stratification is especially simple since the $3$-part partitions
of $6$ are given by $(1,1,4) \le (1,2,3) \le (2,2,2)$.  Hence
\[
\calP_{6,3} = \calP_{6,3_{28}}^{(2,2,2)} \cup
\calP_{6,3_{27}}^{(1,2,3)} \cup \calP_{6,3_{24}}^{(1,1,4)},
\]
and the Zariski closure of each stratum consists of that stratum
together with those to the right of it in the above union.  The small
subscript gives the dimension of each stratum.
\end{example}

\begin{example}
\label{ex93}
When $d = 3$, the smallest $n$ for which incomparable $\bmu$-types
exist is $n = 9$, and the types in question are $(1,4,4)$ and
$(2,2,5)$.  This gives the following stratification of $\calP_{9,3}$:
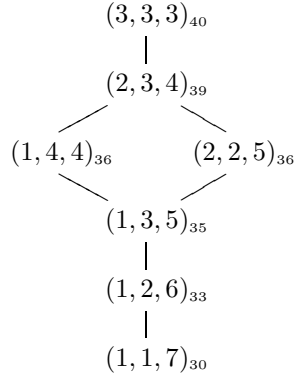
\[
\xymatrix@R=10pt@C=1pt@W=1pt{
& \hskip8pt(3,3,3)_{\scriptscriptstyle40} \ar@{-}[d] &\\ 
& \hskip8pt(2,3,4)_{\scriptscriptstyle39} \ar@{-}[ld] \ar@{-}[rd] &\\
\hskip13pt(1,4,4)_{\scriptscriptstyle36}\hskip-18pt \ar@{-}[rd] && 
\hskip-13pt(2,2,5)_{\scriptscriptstyle36} \ar@{-}[ld]\\
& \hskip8pt(1,3,5)_{\scriptscriptstyle35}\ar@{-}[d] &\\ 
& \hskip8pt(1,2,6)_{\scriptscriptstyle33}\ar@{-}[d] &\\ 
& \hskip8pt(1,1,7)_{\scriptscriptstyle30}&
}
\]
Here, we have written
$\calP_{9,3_{\mathrm{dim}}}^{(\mu_1,\mu_2,\mu_3)}$ more simply as
$(\mu_1,\mu_2,\mu_3)_{\scriptscriptstyle\mathrm{dim}}$, where
``$\mathrm{dim}$'' gives the dimension of the stratum.  The closure of
a stratum consists of the stratum and everything strictly below it in
the diagram.

One consequence of the diagram is that if we move around in
$\calP_{9,3}^{(2,3,4)}$ and reach its boundary, then with high
probability we hit a point in \emph{either} $\calP_{9,3}^{(2,2,5)}$ or
$\calP_{9,3}^{(1,4,4)}$, and these possibilities are equally likely
since both have codimension $3$ in $\overline{\calP}_{9,3}^{(2,3,4)}$.
\end{example}

\subsection{The Largest and Smallest Strata} The stratification of
$\calP_{9,3}$ shown in Example~\ref{ex93} has $\bmu = (3,3,3)$ at the
top:  it is the maximum stratum in the partial order of Definition
\ref{orderdef}  and it is also the unique stratum having the largest
dimension 40. The stratum $\bmu=(1,1,7)$ at the bottom is the minimum
in the partial order of Definition 3.5 and also the unique stratum of
smallest dimension 20. This generalizes as follows. 

\begin{proposition}
\label{bigsmallprop}
Given integers $n \ge d \ge 2$, write $n = kd + r$ where $k,r \in \Z$
and $0 \le r < d$.  Then any $d$-part partition $\bmu$ of $n$
satisfies
\[
\bmu_{\mathrm{min}} = (\underbrace{1,\dots,1}_{d-1},n-d+1) \le \bmu \le
(\underbrace{k,\dots,k}_{d-r},\underbrace{k+1,\dots,k+1}_{r}) =
\bmu_{\mathrm{max}}. 
\] 
Furthermore:
\begin{enumerate}
\item $\dim(\calP_{n,d}^{\sbmu_{\mathrm{max}}}) = (d+1)(n+1)$. 
\item $\bmu_{\mathrm{min}} = \bmu_{\mathrm{max}}$ if and only if $n =
  d$ or $n = d+1$.
\item If $n \ge d+1$, then $\dim(\calP_{n,d}^{\sbmu_{\mathrm{min}}}) =
  d^2+d+2n$ and ${\mathcal P}_{n,d}^{\sbmu_{\mathrm{min}}}$ has
  codimension $(d-1)(n-d-1)$ in $\mathcal{
    P}_{n,d}=\overline{\calP}_{n,d}^{\sbmu_{\mathrm{max}}}$.
\end{enumerate}
\end{proposition}

\begin{proof}
The formulas for $\dim(\calP_{n,d}^{\sbmu_{\mathrm{min}}})$ and
$\dim(\calP_{n,d}^{\sbmu_{\mathrm{max}}})$ follow from
Theorem~\ref{relprimedim}. We omit the rest of the straightforward
proof.
\end{proof}

Applied to $\calP_{9,3}$, this proposition gives $\bmu_{\mathrm{max}}
= (3,3,3)$ and $\bmu_{\mathrm{min}} = (1,1,7)$.  Furthermore,
$\calP_{9,3}^{(1,1,7)}$ has dimension $ 3^2+3+2\cdot 9 = 30$, and its
codimension in $\mathcal P_{n,d}=\overline{\calP}_{9,3}^{(3,3,3)}$ is
$(3-1)(9-3-1) = 10$, in agreement with Example~\ref{ex93}.

We will say more about the structure of
$\calP_{n,d}^{\sbmu_{\mathrm{min}}}$ in Section~\ref{smallest}.

\subsection{The General Case}
\label{generalcase}
We can also allow $a_0,\dots,a_d$ to have a common factor.  Let
\[
\calCP_{n,d} \subseteq R_n^{d+1}
\]
consist of all $(a_0,\dots,a_d) \in R_n^{d+1}$ such that
$a_0,\dots,a_d$ are linearly independent.  This has $\calP_{n,d}$ as
an open subset and in addition contains those parametrizations where
$(a_0,\dots,a_d)$ have a nontrivial common factor.  Recall from
Proposition~\ref{muprop} that the $\bmu$-type of $(a_0,\dots,a_d) \in
\calCP_{n,d}$ satisfies $|\bmu| = n-c$, where
$\deg(\gcd(a_0,\dots,a_d)) = c$.

Given a $d$-part partition $\bmu$ with $|\bmu| \le n$, we have the
$\bmu$-stratum
\[
\calP^\sbmu_{n,d} = \{(a_0,\dots,a_d) \in \calCP_{n,d} \mid
(a_0,\dots,a_d)\ \text{has type }\bmu\}. 
\]  
Since a $d$-part partition satisfies $|\bmu| \ge d$, we will always
assume that $d \le |\bmu| \le n$.  Note also that the $\bmu$-stratum
$\calP^\sbmu_{n,d}$ lies in $\calP_{n,d}$ (i.e., is one of the strata
\eqref{stratadef1}) if and only if $|\bmu| = n$.

\begin{theorem}
\label{gendim} 
Let $\bmu$ be a $d$-part partition satisfying $d \le |\bmu|\le n$. 
Then $\calP_{n,d}^\sbmu$ is a Zariski open subset of
a subvariety of $R_n^{d+1}$.  Furthermore, $\calP_{n,d}^\sbmu$ is
irreducible of dimension
\[
\dim(\calP_{n,d}^\sbmu) =d|\bmu|+d+n+1 - \sum_{i>j}
\max(0,\mu_i-\mu_j-1).
\]
\end{theorem}

We prove this in Appendix~\ref{proofs} using results from \cite{I1,I2}.
Here we give an intuitive argument to explain the formula for
$\dim(\calP_{n,d}^\sbmu)$. Multiplication gives a surjection
\[
(R_{n-|\sbmu|} \setminus \{0\}) \times \calP_{|\sbmu|,d}^\sbmu
\longrightarrow \calP_{n,d}^\sbmu
\]
that maps relatively prime $d$-tuples of degree $|\bmu|$ to $d$-tuples
of degree $n$ with a common factor $h$ of degree $n-|\bmu|$.  The
fibers of this map have dimension $1$ since
\[
(\lambda^{-1}h)(\lambda b_0,\dots,\lambda b_d) = h(b_0,\dots,b_n)
\ \text{for all}\  \lambda \in \sfk \setminus\{0\} 
\]
and $\gcd$'s are well-defined only up to multiplication by a nonzero
constant.  By Theorem~\ref{relprimedim}, it follows that
\[
\dim(\calP_{n,d}^\sbmu) = \big(n-|\bmu|+1\big) + \big((d+1) (|\bmu|+1)
- \sum_{i>j} \max(0,u_i-u_j-1)\big) -1. 
\]
This easily reduces to the formula given in Theorem~\ref{gendim}.\vskip 0.2cm

We can also compute the Zariski closure of $\calP_{n,d}^\sbmu$ in
$\calCP_{n,d}$. 

\begin{theorem}
\label{genclosure}
The Zariski closure of $\calP_{n,d}^\sbmu$ in $\calCP_{n,d}$ is
$\overline{\calP}_{n,d}^\sbmu = \bigcup_{\sbmu' \le \sbmu}
\calP_{n,d}^{\sbmu'}$, where the union is over all $d$-part partitions
$\bmu'$ with $d \le |\bmu'| \le n$ and $\bmu' \le \bmu$.
\end{theorem}

By definition, $\calP_{n,d}^\sbmu$ consists of parametrizations with a
common factor of degree $n-|\bmu|$.  Theorem~\ref{genclosure} tells us
that in $\calCP_{n,d}$, the difference  $\overline{\calP}_{n,d}^\sbmu
\setminus \calP_{n,d}^\sbmu$ consists of strata
$\overline{\calP}_{n,d}^{\sbmu'}$, where $\bmu' \le \bmu$.  Since
$\bmu' \le \bmu$ implies that $n-|\bmu'| \ge n-|\bmu|$, we see that
$\overline{\calP}_{n,d}^\sbmu$ may include strata with common factors
of larger degree.   

\begin{example}
In Example~\ref{ex63}, we saw that $\calP_{6,3}$ has a simple
stratification with three strata.  These and four other strata appear
in the more complex stratification of $\calCP_{6,3}$:
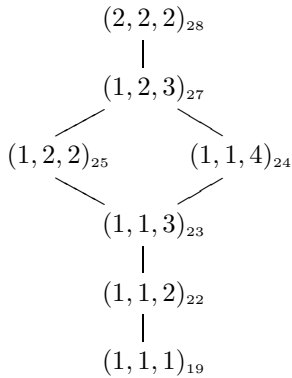
\[
\xymatrix@R=10pt@C=1pt@W=1pt{
& \hskip8pt(2,2,2)_{\scriptscriptstyle28} \ar@{-}[d] &\\ 
& \hskip8pt(1,2,3)_{\scriptscriptstyle27} \ar@{-}[ld] \ar@{-}[rd] &\\
\hskip13pt(1,2,2)_{\scriptscriptstyle25}\hskip-18pt \ar@{-}[rd] && 
\hskip-13pt(1,1,4)_{\scriptscriptstyle24} \ar@{-}[ld]\\
& \hskip8pt(1,1,3)_{\scriptscriptstyle23}\ar@{-}[d] &\\ 
& \hskip8pt(1,1,2)_{\scriptscriptstyle22}\ar@{-}[d] &\\ 
& \hskip8pt(1,1,1)_{\scriptscriptstyle19}&
}
\]
Here, we write $\calP_{6,3_{\mathrm{dim}}}^{(\mu_1,\mu_2,\mu_3)}$ as
$(\mu_1,\mu_2,\mu_3)_{\scriptscriptstyle\mathrm{dim}}$, similar to
Example~\ref{ex93}.

This diagram tells us that if we move around in
$\calP_{6,3}^{(1,2,3)}$ and reach its boundary in $\calCP_{6,3}$, then
with high probability we hit a point in \emph{either}
$\calP_{6,3}^{(1,2,2)}$ (acquire a common factor) or
$\calP_{6,3}^{(1,1,4)}$ (remain relatively prime).  These
possibilities are not equally likely since the former has codimension
$2$ in $\overline{\calP}_{6,3}^{(1,2,3)}$ while the latter has codimension $3$.
\end{example}

\section{Further Topics}
\label{further}

In this section we investigate non-proper parametrizations and look
more closely at the smallest and largest strata of $\calP_{n,d}$.  Section \ref{nonproper}
is new to this paper;  Section \ref{smallest} is a geometric interpretation of 
some results of \cite{I2}.

\subsection{Non-Proper Parametrizations}\label{nonproper}

In this section, we restrict our attention to $\bmu$-strata
$\calP_{n,d}^\sbmu$ with $|\bmu| = n$, i.e., $\bmu$ is a $d$-part
partition of $n$.  This means that all $(a_0,\dots,a_d) \in
\calP_{n,d}^\sbmu$ are relatively prime.  As in Section~\ref{higher},
we assume $n \ge d \ge 2$.

A parametrization $(a_0,\dots,a_d) : \PP^1 \to \PP^d$ is proper if it
is birational onto its image.  In general, let $k$ be number of points
in the preimage of a generic point in the image.  We say that the
parametrization has \emph{generic degree $k$}.  It is well-known that
if the $a_i \in R_n$ are relatively prime, then
\begin{equation}
\label{nkm}
n = km,
\end{equation}
where $k$ is the generic degree of the parametrization and $m$ is the
degree of the image curve $C \subseteq \PP^d$.  Thus a proper
parametrization has generic degree $1$ and parametrizes a curve of
degree $n$.

\begin{proposition}
\label{nonproperprop1}
Let $k > 1$ be an integer.  Then $\calP_{n,d}$ contains a
parametrization of generic degree $k$ if and only if $k \mid n$ and $n
\ge kd$.
\end{proposition}

\begin{proof} 
The restriction $k \mid n$ is obvious from \eqref{nkm}.  To understand
the inequality $n \ge kd$, recall our assumption that the $a_i$ are
linearly independent, i.e., the image curve $C$ does not lie in any
hyperplane of $\PP^d$.  But $C$ has degree $m = n/k$, and if $m < d$,
then the $d+1$ polynomials of degree $m$ parametrizing $C$ would have
to be linearly dependent, forcing $C$ to lie in a hyperplane.  Thus
$n/k = m \ge d$.  This proves one direction of the proposition; the
proof of the other direction will be deferred until
Section~\ref{nonproperproofs}.
\end{proof}

We next relate non-proper parametrizations to the
$\bmu$-stratification of $\calP_{n,d}$.

\begin{proposition}
\label{nonproperprop}
Suppose we have a $\bmu$-stratum $\calP_{n,d}^\sbmu$ with $\bmu =
(\mu_1,\dots,\mu_d)$ and let $k > 1$ be an integer.  Then
$\calP_{n,d}^\sbmu$ contains parametrizations of generic degree $k$ if
and only if $k \mid \mu_i$ for all $i$.
\end{proposition}

The proof will be given in Section~\ref{nonproperproofs}.
Proposition~\ref{nonproperprop} has the following useful corollary.

\begin{corollary}
\label{nonpropercor}
Given $\bmu = (\mu_1,\dots,\mu_d)$, the $\bmu$-stratum
$\calP_{n,d}^\sbmu$ consists entirely of proper parametrizations if
and only if $\gcd(\mu_1,\dots,\mu_d) = 1$.
\end{corollary}

\begin{example}
\label{ex123}
Suppose that $n = 12$ and $d = 3$.  The integers $k > 1$ dividing $12$
are $k = 2,3,4,6,12$.  By Proposition~\ref{nonproperprop1},
$\calP_{12,3}$ has no parametrizations of generic degree $k = 6$ or
$12$ since $d = 3$, while $k = 2$, $3$ and $4$ can occur.

One can compute that $\calP_{12,3}$ decomposes into $12$ $\bmu$-strata,
eight of which satisfy the gcd criterion of Corollary~\ref{nonpropercor}
and hence have no non-proper parametrizations.  For the remaining four
$\bmu$-strata, we have non-proper parametrizations of the following types:
\begin{itemize}
\item Generic degree $4$ occurs in
  $\calP_{12,3}^{(4,4,4)}$. 
\item Generic degree $3$ occurs  in
  $\calP_{12,3}^{(3,3,6)}$. 
\item Generic degree $2$ occurs  in
  $\calP_{12,3}^{(4,4,4)}$, $\calP_{12,3}^{(2,4,6)}$, and
  $\calP_{12,3}^{(2,2,8)}$.  
\end{itemize}
\end{example}

One expects non-proper parametrizations to be rare.  Our next task is
to quantify this intuition by computing the size of the generic degree
$k$ locus in each $\bmu$-stratum.  When $\calP_{n,d}^\sbmu$ contains a
parametrization of generic degree $k > 1$,
Proposition~\ref{nonproperprop} implies that its $\bmu$-type can be
written as $\bmu = k(\tilde\mu_1,\dots,\tilde\mu_d)$, i.e., $\bmu$ is
divisible by $k$.  This implies $k \mid n$ since $n =
k\tilde\mu_1+\cdots+k\tilde\mu_d$.

\begin{theorem}
\label{nonproperthm}
Assume that $\bmu$ is divisible by $k > 1$.  Then the parametrizations
in $\calP_{n,d}^{\sbmu}$ of generic degree $k$ form a nonempty
constructible subset of $\calP_{n,d}^{\sbmu}$ with irreducible Zariski
closure of codimension
\begin{equation} \label{codnonpropereq}
(k-1)(m(d+1) - S - 2),\quad 
S = \sum_{i > j} \max(0,\tilde\mu_i - \tilde\mu_j),
\end{equation}
where $m = n/k$ and $\bmu = k(\tilde\mu_1,\dots,\tilde\mu_d)$. Furthermore: 
\begin{enumerate}
\item The codimension is at least $(k-1)(d(d-1)+2m-2)$ and is always positive.
\item A generic parametrization in $\calP_{n,d}^{\sbmu}$ is proper.
\end{enumerate}
\end{theorem}

The proof of this theorem will be given in
Section~\ref{nonproperproofs}.  Here is a sketch of some of the ideas
involved.  Using L\"uroth's Theorem, we will show that a
parametrization $F : \PP^1 \to \PP^d$ of generic degree $k$ is a
composition
\[
\PP^1 \stackrel{G}{\longrightarrow} \PP^1
\stackrel{H}{\longrightarrow} \PP^d,
\]
where $G$ is defined by $(\alpha(s,t),\beta(s,t)) \in R_k \times R_k$
of degree $k$ and $H$ is a parametrization in
$\calP_{m,d}^{\tilde\sbmu}$ of generic degree $1$ for $\tilde\bmu =
(\tilde\mu_1,\dots,\tilde\mu_d)$.  It will follow that composition
gives a map
\[
\varphi: \calP_{m,d}^{\tilde\sbmu} \times R_k \times R_k \longrightarrow
\calP_{n,d}^{\sbmu} 
\]
(we have to shrink the domain slightly to make this work) whose image
consists of parametrizations of generic degree $k$.  The
nonempty fibers of this map have dimension $4$ coming from a natural
action of $\mathrm{GL}(2,\sfk)$.  Hence the generic degree $k$ locus
has dimension
\[
\dim(\calP_{m,d}^{\tilde\sbmu}) + 2(k+1) - 4 = (d+1)(m+1) -
\sum_{i>j}\max(0,\tilde\mu_i-\tilde\mu_j-1) +2(k-1).
\]
where we have used Theorem~\ref{relprimedim}.  The codimension formula
in Theorem~\ref{nonproperthm} follows by combining this with the
corresponding formula for $\dim(\calP_{n,d}^{\sbmu})$.  Full details
will be provided in Section~\ref{nonproperproofs}.

\begin{example}
\label{ex93nonproper}
In Example~\ref{ex93}, we saw that the maximum $\bmu$-stratum of
$\calP_{9,3}$ in the order of Definition \ref{orderdef} is $\calP_{9,3}^{(3,3,3)}$ of dimension $40$ and the
minimum is $\calP_{9,3}^{(1,1,7)}$ of dimension $20$.  The only
non-proper parametrizations have generic degree $3$ and lie in
$\calP_{9,3}^{(3,3,3)}$.  Since $\tilde\bmu = (1,1,1)$ and $m = n/k =
3$ in this case, Theorem~\ref{nonproperthm} implies that the generic
degree $3$ locus has codimension
\[
(k-1)(m(d+1)-S-2) = (3-1)(3\cdot(3+1) - 0 - 2) = 2\cdot 10 = 20
\]
in $\calP_{9,3}^{(3,3,3)}$.  Hence non-proper parametrizations really
are rare!  Note also that
\[
20 = (3-1)(3(3-1)+2\cdot3-2) = (k-1)(d(d-1)+2m-2)
\]
since $k = d = m = 3$.  This shows that the lower bound in
Theorem~\ref{nonproperthm}(1) is sharp.
\end{example}
\begin{example}\label{12,3nonproperex}
In Example \ref{ex123} we noted that non-proper parametrizations with
generic degrees $k=4,3$ and $2$ occur in $\mathcal P_{12,3}$.  We give
the respective codimensions using the formula \eqref{codnonpropereq}
of Theorem \ref{nonproperthm}:
\begin{itemize}
\item Generic degree $4$ occurs in $\calP_{12,3}^{(4,4,4)}$, which has
  dimension $52$.  Here, we have $m=3$ and $S=0$, so that the non-proper
  codimension is $(4-1)(3\cdot 4-2)=30$.  In other words, the non-proper
  parametrizations have dimension $22$.
\item Generic degree $3$ occurs in $\calP_{12,3}^{(3,3,6)}$, which has
  dimension $48$.  Here, $m=4$ and $S=2$ so the non-proper codimension
  is $(3-1)(4\cdot 4-2-2)=24$.  Hence the non-proper locus has dimension
  $24$.
\item Generic degree $2$ occurs in $\calP_{12,3}^{(4,4,4)}$,
  $\calP_{12,3}^{(2,4,6)}$, and $\calP_{12,3}^{(2,2,8)}$.  A similar
  calculation gives respective non-proper codimensions $22$, $18$, and
  $18$ in $\bmu$-strata of dimensions $52$, $47$, and $42$.  So these
  non-proper strata have dimensions $30$, $29$ and $24$, respectively.
  Again, the non-proper loci have very high codimensions.
\end{itemize}

\end{example}

\subsection{The Smallest Stratum and Rational Normal Scrolls}
\label{smallest}

In this section we describe a further stratification of the
smallest $\bmu$-stratum $\bmu_{\mathrm{min}} = (1,\dots,1,n-d+1)$ in
the relatively prime case, assuming $n \ge d+1$.  The stratification
will involve finding which rational normal scroll the curve lies on.
We begin with an example.

\begin{example}
\label{WJGexample}
Rational curves in $\PP^3$ with $\bmu = (1,1,n-2)$ were studied in
\cite{WJG}.  Corollary 6.8 of that paper uses $\bmu$-bases to show
that when $n \ge 4$, such curves are either smooth or have a unique
singular point of multiplicity $n-2$.

This result can be explained using Section 2 of the paper \cite{KPU},
which considers $\bmu_{\mathrm{min}} = (1,\dots,1,n-d+1)$ from a
commutative algebra point of view.  When $\bmu = (1,1,n-2)$, the
results of \cite{KPU} imply that after a suitable change of
coordinates in $\PP^3$, the $4\times 3$ matrix $A$ from
\eqref{genexact} can be assumed to be either
\begin{equation}
\label{KPU1}
A = \begin{pmatrix} s & 0 & r_0\\ t & 0 & r_1\\ 0 & s &
  r_2\\ 0 & t & r_3\end{pmatrix},\quad \deg(r_i) = n-2, 
\end{equation}
or 
\begin{equation}
\label{KPU2}
A = \begin{pmatrix} s & 0 & r_0\\ t & s & r_1\\ 0 & t &
  r_2\\ 0 & 0& r_3\end{pmatrix},\quad \deg(r_i) = n-2. 
\end{equation}
(See \cite[Prop.\ 2.1]{KPU}.) In either case, the first column of $A$
gives the moving plane $sx_0 + tx_1 = 0$ which contains the line $L_1
= \{(0,0,a,b) \mid (a,b) \in \PP^1\}$.  In the terminology of
\cite{WJG}, this is an \emph{axial moving plane} with $L_1$ as axis.
The second column of $A$ also gives an axial moving plane with axis
$L_2$, the difference being that in \eqref{KPU1}, the axes $L_1$ and $L_2$ are
disjoint (and the curve is smooth), while in \eqref{KPU2}, the axes
meet at $(0,0,0,1)$ (and the curve is singular at this point).   

Let us look at \eqref{KPU1} more closely.  Recall that the
parametrization $(a_0,a_1,a_2,a_3)$ is given by the signed maximal
minors of $A$.  If we set $h_1 = sr_3-tr_2$ and $h_2 = sr_1-tr_0$,
then one easily computes that
\begin{equation}
\label{KPUparam1}
a_0 = t\hskip.5pt h_1,\ a_1 = -s\hskip.5pt h_1,\ a_2 = -t\hskip.5pt
h_2,\ a_3 = s\hskip.5pt h_2. 
\end{equation}
Note that $h_1$ and $h_2$ have degree $n-1$.  

A first consequence of \eqref{KPUparam1} is that $a_0a_3 = a_1a_2$, so
that the curve lies on the smooth quadric surface $x_0x_3 = x_1x_2$ in
$\PP^3$.  We will see that this surface is a particularly simple
example of a \emph{rational normal scroll}.  In terms of ideals,
\eqref{KPUparam1} implies that
\[
I = \langle a_0,a_1,a_2,a_3\rangle = \langle
t\hskip.5pt h_1,s\hskip.5pt h_1,t\hskip.5pt h_2,s\hskip.5pt h_2\rangle
= \langle h_1,h_2\rangle \cap 
{\textstyle\bigoplus_{m = n}^\infty} R_m. 
\]
Here, recall that $R = \sfk[s,t]$, so that the above equation tells us
that $I$ consists of all elements of degree $m \ge n$ in the
simpler ideal $\langle h_1,h_2 \rangle$.   In the terminology of
\cite{I2}, $\langle h_1,h_2\rangle$ is the \emph{ancestor ideal} of
$I$. 

For \eqref{KPU2}, we have a similar situation.  Let $(a_0,a_1,a_2,a_3)$
be the parametrization coming from \eqref{KPU2} and set $h_1 = r_3$ and
$h_2 = a_3$.  Taking the first three maximal minors of $A$, we have 
\begin{equation}
\label{KPUparam2}
a_0 = t^2\hskip.5pt h_1,\ a_1 = -st\hskip.5pt h_1,\ a_2 =
s^2\hskip.5pt h_1,\ a_3 = h_2. 
\end{equation}
Here $h_1$ has degree $n-2$ and $h_2$ has degree $n$.   

{}From \eqref{KPUparam2} we see that $a_0a_2 = a_1^2$, so that the
curve lies on the singular quadric surface $x_0x_2 = x_1^2$ in
$\PP^3$.  This surface is another example of a rational normal scroll.
In terms of ideals, \eqref{KPUparam2} implies that
\[
I = \langle a_0,a_1,a_2,a_3\rangle = \langle
t^2\hskip.5pt h_1,st\hskip.5pt h_1,s^2\hskip.5pt h_1,h_2\rangle =
\langle h_1,h_2\rangle \cap 
{\textstyle\bigoplus_{m = n}^\infty} R_m. 
\]
Again, $\langle h_1,h_2\rangle$ is the ancestor ideal of $I$.
\end{example}

This example shows that the $\bmu$-stratum for $(1,1,n-2)$ decomposes
into two parts corresponding to \eqref{KPU1} and \eqref{KPU2}, each
of which has a rational normal scroll and an ancestor ideal.  There
is some rich geometry and algebra going on here.   

In general, the \emph{ancestor ideal} of $I = \langle
a_0,\dots,a_d\rangle\subseteq R$ is the largest homogeneous ideal of
$R$ that equals $I$ in degrees $n$ and higher (remember that
$a_0,\dots,a_d \in R_n$).  Here is a result from \cite{I2}, whose
proof we defer until Section~\ref{smallestproofs}.

\begin{theorem}
\label{ancestor1}
Let $I = \langle a_0,\dots,a_d\rangle\subseteq R$ have $\bmu$-type
$\bmu_{\mathrm{min}} = (1,\dots,1,n-d+1)$, $n \ge d+1$.  Then the
ancestor ideal of $I$ is equal to $\langle h_1,h_2\rangle$, where
$h_1,h_2$ are relatively prime and satisfy
\[
I_n = R_{n-\deg(h_1)}\cdot h_1 \oplus R_{n-\deg(h_2)} \cdot h_2.
\]
\end{theorem}

Since $I_n$ has vector space dimension $d+1$, this theorem implies that
\[
d+1 = \alpha_1 + \alpha_2, \quad \alpha_i = n+1-\deg(h_i) \ge 1.
\]
If we assume $\deg(h_1) \ge \deg(h_2)$, then $\alpha_1 \le \alpha_2$,
so that we have the partition
\[
\mathcal{A} = (\alpha_1,\alpha_2)
\]
of $d+1$.  

This partition determines the \emph{rational normal scroll}
$S_{\alpha_1-1,\alpha_2-1} \subseteq \PP^d$, which consists of the
points
\[
\lambda \hskip.5pt(s^{\alpha_1-1},s^{\alpha_1-2}t \dots,
t^{\alpha_1-1},0,\dots,0) + \mu\hskip.5pt(0,\dots, 0, s^{\alpha_2-1},
s^{\alpha_1-2}t\dots,t^{\alpha_2-1})
\]
for all $(s,t), (\lambda,\mu)$ in $\PP^1$.  In this formula, the first
expression in parentheses is the \emph{rational normal curve} of
degree $\alpha_1-1$, sitting in the first $\alpha_1$ coordinates of
$\PP^d$; the second expression in parentheses is the rational normal
curve of degree $\alpha_2-1$, sitting in the last $\alpha_2$
coordinates of $\PP^d$.  Note how this uses $\alpha_1+\alpha_2 = d+1$.
These two rational normal curves are the ``edges" of the scroll, which
consists of lines parametrized by $(\lambda,\mu)$ that connect the
points on the two edges with the same parameter value $(s,t)$.  It is
well-known that $S_{\alpha_1-1,\alpha_2-1}$ is a surface of degree
$d-1$ in $\PP^d$ (see \cite{EH}).

In our partition, we assume $1 \le \alpha_1 \le \alpha_2$.  When
$\alpha_1 = 1$, the ``rational normal curve of degree 0'' is just a
point $(1,0,\dots,0)$.  Since $\alpha_1+\alpha_2 = d+1$, we have
$\alpha_2 = d$, so that the rational normal scroll
$S_{\alpha_1-1,\alpha_2-1} = S_{0,d-1}$ is just the cone over the
rational normal curve of degree $d-1$ sitting in the last $d$
coordinates of $\PP^d$.

To see how this scroll relates to the paramerization given by $I =
\langle a_0,\dots,a_d\rangle$, we use Theorem~\ref{ancestor1} to write
the ideal as 
\[
I = \langle s^{\alpha_1-1}\hskip.5pt h_1, \dots,
t^{\alpha_1-1}\hskip.5pt h_1,s^{\alpha_2-1}\hskip.5pt
h_2,\dots,t^{\alpha_2-1}\hskip.5pt h_2\rangle.
\]
Switching to these generators of $I$ corresponds to a coordinate
change in $\PP^d$. For $(s,t) \in \PP^1$, the parametrization gives
the point 
\begin{equation}
\label{onscroll}
h_1(s,t)(s^{\alpha_1-1},\dots, t^{\alpha_1-1},0,\dots,0) +
h_2(s,t)(0,\dots,0,s^{\alpha_2-1},\dots,t^{\alpha_2-1}), 
\end{equation}
which is clearly on $S_{\alpha_1-1,\alpha_2-1}$.  Hence we have
proved:

\begin{corollary}
\label{ancestorcor}
Let $I = \langle a_0,\dots,a_d\rangle\subseteq R$ have $\bmu$-type
$\bmu_{\mathrm{min}} = (1,\dots,1,n-d+1)$, $n \ge d+1$.  If the
ancestor ideal of $I$ gives the partition $\mathcal{A} =
(\alpha_1,\alpha_2)$ of $d+1$, then after a change of coordinates in
$\PP^d$, the parametrized curve lies on the rational normal scroll
$S_{\alpha_1-1,\alpha_2-1}$.
\end{corollary}

\begin{example}
\label{WJGexagain}
When $d = 3$ and $\bmu_{\mathrm{min}} = (1,1,n-2)$, the only two
partitions of $4$ are $4 = 2+2 = 3+1$.  The corresponding rational
normal scrolls are $S_{1,1}$, defined by $x_0x_3=x_1x_2$, and
$S_{0,2}$, defined by $x_1x_3=x_2^2$.  Hence we recover (after a small
change of coordinates) the two quadric surfaces encountered in
Example~\ref{WJGexample}. 
\end{example}

Finally, fix a partition $\mathcal{A} = (\alpha_1,\alpha_2)$ of $d+1$
with $1 \le \alpha_1 \le \alpha_2$.  Then define
$\calP_{n,d,\mathcal{A}}^{\sbmu_{\mathrm{min}}}\subseteq
\calP_{n,d}^{\sbmu_{\mathrm{min}}}$ to be the subset consisting of
\emph{all} parametrizations in the stratum whose ancestor ideal gives
the partition $\mathcal{A}$.  Recall from
Proposition~\ref{bigsmallprop} that
$\dim(\calP_{n,d}^{\sbmu_{\mathrm{min}}}) = d^2+d+n$ when $n \ge d+1$.
We defer the proof of the following result until
Section~\ref{smallestproofs}.

\begin{theorem}
\label{ancestor2}
The subsets $\calP_{n,d,\mathcal{A}}^{\sbmu_{\mathrm{min}}}\subseteq
\calP_{n,d}^{\sbmu_{\mathrm{min}}}$ have the following properties: 
\begin{enumerate}
\item  $\calP_{n,d,\mathcal{A}}^{\sbmu_{\mathrm{min}}}$ is irreducible
  and is open in its Zariski closure in
  $\calP_{n,d}^{\sbmu_{\mathrm{min}}}$ 
\item $\dim(\calP_{n,d,\mathcal{A}}^{\sbmu_{\mathrm{min}}}) = d^2+d+2n
  - \max(0,a_2-a_1-1)$. 
\item The Zariski closure
  $\overline{\mathcal{P}}^{\sbmu_{\mathrm{min}}}_{n,d,\mathcal A}$ in
  $\mathcal{P}^{\sbmu_{\mathrm{min}}}_{n,d}$ satisfies
\[
\overline{\mathcal{P}}^{\sbmu_{\mathrm{min}}}_{n,d,\mathcal A}
=\bigcup_{\mathcal A'\le \mathcal A}\mathcal
P^{\sbmu_{\mathrm{min}}}_{n,d,\mathcal A'}, 
\]
where the union is over $2$-part partitions $\mathcal A'$ of $d+1$. 
\end{enumerate}
\end{theorem}

\begin{example}
\label{WJGfinal}
Continuing our study of $(1,1,n-2)$ with $n \ge 4$, we have the
partitions $(1,3) \le (2,2)$.  The larger partition $(2,2)$ is generic
and gives an open dense subset of $\calP_{n,3}^{(1,1,n-2)}$, of dimension
$2n+12$.  Parametrizations with the smaller partition $(1,3)$ lie in a
closed subset of codimension 1 since $\max(0,3-1-1) = 1$. 
\end{example}

In the above example, recall from \cite[Cor.\ 6.8]{WJG} that the
curves corresponding to $(2,2)$ are smooth while those for $(1,3)$
have a unique singular point of multiplicity $n-2$.  It would be
interesting to study the singularities of curves in $\mathcal
P^{\sbmu_{\mathrm{min}}}_{n,d,\mathcal A}$ in the general case when
$\bmu_{\mathrm{min}} = (1,\dots,1,n-d+1)$. 

We conclude by noting that \emph{all} ideals coming from
parametrizations, not just those in the smallest stratum, have
ancestor ideals that determine rational normal scrolls (possibly of
high dimension) containing the curve.  More precisely, suppose we have
$I = \langle a_0,\dots,a_d\rangle$, where $a_0,\dots,a_d \in R_n$ are
linearly independent and relatively prime.  If $I$ has $\bmu$-type
$\bmu = (\mu_1,\dots,\mu_d)$, then we will see in
Appendix~\ref{proofs} that the ancestor ideal of $I$ can be written
\[
\langle h_1,\dots,h_\tau\rangle,
\]
where 
\begin{equation}
\label{tauformula}
\tau = d+1-\#\{i \mid \mu_i=1\}
\end{equation}
and 
\begin{equation}
\label{Indecomp}
I_n = R_{n-\deg h_1}\cdot  h_1\oplus R_{n-\deg h_2}\cdot h_2\oplus
\cdots \oplus R_{n-\deg h_\tau}  \cdot h_\tau. 
\end{equation}
Note that $\tau = 2$ precisely when $\bmu$ has $d-1$ indices with
$\mu_i = 1$, i.e., when $\bmu = (1,\dots,1,n-d+1)$ and $n \ge d+1$.   

The decomposition \eqref{Indecomp} of $I_n$ gives the partition
\begin{equation}
\label{dp1partition}
d+1=\sum_{i=1}^{\tau} \alpha_i,\quad \alpha_i = n+1-\deg h_i.
\end{equation}
Setting $\mathcal{A} = (\alpha_1,\dots,\alpha_\tau)$ determines a
subset $\calP_{n,d,\mathcal{A}}^{\sbmu}$, which as we will see in
Section~\ref{smallestproofs} has codimension and closure properties
similar to those stated in Theorem~\ref{ancestor2}.

The partition \eqref{dp1partition} gives a rational normal scroll
$S_{\alpha_1-1,\dots,\alpha_\tau-1} \subseteq \PP^d$.  This is formed
by putting $\tau$ rational normal curves in $\PP^d$ using the first
$\alpha_1$ coordinates for the first curve, the next $\alpha_2$
coordinates for the second, and so on.  This works since the
$\alpha_i$ partition $d+1$.  As in the surface case considered
earlier, the ``rational normal curve'' reduces to a point when
$\alpha_i = 1$.

The $\tau$ rational normal curves are the ``edges'' of the scroll.
For a fixed parameter value $(s,t)$, we get $\tau$ points, one on each
of the $\tau$ curves.  These points determine a subspace of dimension
$\tau-1$.  Then $S_{\alpha_1-1,\dots,\alpha_\tau-1}$ is the union of
these subspaces as we vary $(s,t)$ over $\PP^1$.

When $(a_0,\ldots ,a_d)$ have no common factor, it is not hard to show
that
\[
2\le \tau\le \min\{d+1, n+1-d\}.
\]
If $\tau=d+1$, it is easy to see that there is a unique partition
$\mathcal{A} = (1, \ldots , 1)$ and $S_{0,\dots,0}$ is just $\PP^d$;
and if $\tau=d$, then the unique partition is $\mathcal{A} =
(1,\ldots, 1,2)$ and we again get $\PP^d$.  So the interesting cases
are when $\tau \le d-1$.  Here $S_{\alpha_1-1,\dots,\alpha_\tau-1}$
has dimension $\tau$, and one can show that its degree is
\begin{equation}
\label{Sdegree}
\deg(S_{a_1-1,\dots,a_\tau-1}) = d+1-\tau.
\end{equation}
See \cite{EH} for more on rational normal scrolls.

When $I$ has ancestor ideal $\langle h_1,\dots,h_\tau\rangle$,
\eqref{onscroll} generalizes to show that the corresponding curve lies
on $S_{a_1-1,\dots,a_\tau-1}$.  By \eqref{tauformula} and
\eqref{Sdegree}, the degree of the scroll equals $\#\{i \mid \mu_i =
1\}$.  Thus the number of $1$'s in the $\bmu$-type of the
parametrization determines the dimension and degree of the scroll
containing the curve.  Note also that the interesting case $\tau \le
d-1$ occurs only when $\#\{i \mid \mu_i = 1\} \ge 2$.  In $\PP^3$,
these are the $\bmu$-types $(1,1,n-2)$ considered in
Examples~\ref{WJGexample}, \ref{WJGexagain}, and \ref{WJGfinal}.

Finally, we should mention that the rational normal scrolls discussed
here are closely related to (but not the same as) the scrolls
considered in \cite{KPU}. They study $\bmu = (1,...,1,n-d+1)$, so
$\tau = 2$.  In this case, our rational normal scroll is the surface
$S_{a_1-1,a_2-2} \subseteq \mathbb P^d$.  In [KPU], they work in
$\mathbb P^{d+2}$ with homogeneous coordinates $x_0,\dots,x_d,s,t$ and
consider the three-dimensional scroll $S_{a_1-1,a_2-2,1} \subseteq
\mathbb P^{d+2}$.

\section*{Acknowlegments} 
We would like to thank the referees for their comments and careful
reading of the paper.  The second author is grateful to Claudia Polini
and others for conversations at the ALGA conference at IMPA in 2012.
These conversations led to \cite{I3} and, serendipitously, to this paper.

\appendix

\section{Proofs of the Main Results}
\label{proofs}

\small

Most proofs omitted in Section \ref{higher} of the paper can be found
in \cite{I2}.  However, \cite{I2} is written in the language of
Grassmannians, ancestor ideals, and Hilbert functions, so some
translation is needed to the situation of this paper.  Section
\ref{higherproofs} of the appendix is for experts who want to make
sure that nothing has been lost in translation. Section
\ref{nonproperproofs} proves the results of Section \ref{nonproper};
Section \ref{smallestproofs} proves the results of Section
\ref{smallest} and gives some extensions of those results to
non-minimal $\bmu$-strata.

\subsection{Notation} 
We have worked with parametrizations $(a_0,\dots,a_d) \in R_n^{d+1}$
and their associated ideals $I = \langle a_0,\dots,a_d\rangle
\subseteq R$.  In \cite{I2}, the focus is on subspaces $V \subseteq
R_n$ of dimension $d+1$, and (in the notation of \cite{I2}) the ideal
$(V) \subseteq R$ generated by $V$.  One can translate between
\cite{I2} and this paper via
\begin{align*}
V &\longleftrightarrow \mathrm{Span}(a_0,\dots,a_d) \\
(V) &\longleftrightarrow  \langle a_0,\dots,a_d\rangle.
\end{align*}

Subspaces $V \subseteq R_n$ of dimension $d+1$  correspond to
elements of the Grassmannian $\Grass(d+1,R_n)$.  The Hilbert function
of the graded $\sfk$-algebra $R/(V)$ will be denoted $H_V$.   Thus
\[
H_V(m) = \dim_{\sfk}((R/(V))_m)
\]
for $m \ge 0$.  We say $H_V\le H_{V'}$ if $H_V(m) \le H_{V'}(m)$ for
all $m\ge 0$.   
  
\subsection{Proofs for Section~\ref{higher}}\label{higherproofs}
Recall that $\mathcal {CP}_{n,d}$ consists of linearly independent
$(d+1)$-tuples $(a_0,\dots,a_d) \in R_n^{d+1}$.  This can be regarded
as the set of all possible ordered bases of elements of
$\Grass(d+1,R_n)$.  In particular, we have a map 
\[
\pi :\mathcal {CP}_{n,d}\longrightarrow \Grass (d+1,R_n)
\]
defined by $\pi(a_0,\dots,a_n)=\mathrm{Span}(a_0,\ldots ,a_d)
\subseteq R_n$.  We denote by $\Grass^\sbmu_{n,d}$ the image $\pi
(\mathcal {P}^\sbmu_{n,d})$, where $\calP^\sbmu_{n,d} \subset
\mathcal{CP}_{n,d}$ is defined in Section~\ref{generalcase}.

Any two ordered bases of $V$ are related by a unique element of the
general linear group $\mathrm{GL}(d+1,\sfk)$.  Note also that
$\mathrm{GL}(d+1,\sfk)$ is a Zariski open subset of
$\Mat_{d+1}(\sfk)$, which is an affine space of dimension $(d+1)^2$.
Hence we get the following result that relates parametrizations to
subspaces.

\begin{lemma}
\label{coverlem}
 The projection $\pi$ makes $\mathcal {CP}_{n,d}$ into a locally
 trivial bundle over $\Grass(d+1,R_n)$ with fibre isomorphic to
 $\mathrm{GL}(d+1,\sfk)$.  Thus:
 \begin{enumerate}
\item We have an equality of codimensions
 \[
 \cod(\mathcal {P}^\sbmu_{n,d} \subseteq \mathcal {CP}_{n,d})
 =\cod(\Grass^\sbmu_{n,d} \subseteq \Grass (d+1,R_n)). 
 \]
 \item The Zariski closures of $\mathcal{P}^\sbmu_{n,d}$ and 
 $\Grass^\sbmu_{n,d}$ are related by
 \[
\overline{\mathcal {P}}^\sbmu_{n,d}=\pi^{-1}
(\overline{\Grass}^\sbmu_{n,d}).
 \]
 \end{enumerate}
 \end{lemma}
 
Fix $(d+1)$-dimensional subspaces $V$ and $V'$ of $R_n$.  Since the
$\bmu$-type depends only on the ideal (see Proposition~\ref{muprop}),
the ideals $(V)$ and $(V')$ have respective
$\bmu$-types $\bmu$ and $\bmu'$.  These ideals also have Hilbert
functions $H_V$ and $H_{V'}$.  We need the following comparison
result.

\begin{lemma}
\label{mucomparisonlem}
Suppose $V$ and $V'$ are $(d+1)$-dimensional subspaces of $R_n$ with
respective $\bmu$-types $\bmu$ and $\bmu'$ and Hilbert functions  
$H_V$ and $H_{V'}$.  Then:
\begin{enumerate}
\item $H_V(m) = n - |\bmu|$ for $m \gg 0$.
\item $\bmu' \le \bmu$ if and only if $H_{V'} \ge H_V$.
\item $\bmu' = \bmu$ if and only if $H_{V'} = H_V$.
\end{enumerate}
\end{lemma}

\begin{proof}
We first study $H_V$.  Let $\bmu = (\mu_1,\dots,\mu_d)$ with $\mu_i
\le \cdots \le \mu_d$.  By \eqref{genexact}, the ideal $I = 
(V) \subseteq R$ gives the free resolution
\[
0 \longrightarrow {\textstyle\bigoplus_{i=1}^d} R(-n-\mu_i)
\longrightarrow R(-n)^{d+1}
\longrightarrow R \longrightarrow R/I
\longrightarrow 0.
\]
In degree $m$, this becomes
\[
0 \longrightarrow {\textstyle\bigoplus_{i=1}^d} R_{m-n-\mu_i}
\longrightarrow R_{m-n}^{d+1}
\longrightarrow R_m \longrightarrow (R/I)_m
\longrightarrow 0.
\]
Thus 
\begin{equation}
\label{HV}
\begin{aligned}
H_V(m) &= \dim_{\sfk}( (R/I)_m)\\[-10pt] 
&= \dim_{\sfk}(R_m) - (d+1)\dim_{\sfk}(R_{m-n}) +
\sum_{i=1}^d \dim_{\sfk}(R_{m-n-\mu_i}).
\end{aligned}
\end{equation}

Since $\dim_{\sfk}(R_\ell) = \max(0,\ell+1)$ for all $\ell \in
\mathbb{Z}$, an easy computation using \eqref{HV} shows that $H_V(m) =
n - |\bmu|$ for $m \gg 0$.  This proves part (1) of the lemma.

For part (2), set $G_V(m) = \sum_{i=1}^d \dim_{\sfk}(R_{m-n-\mu_i})$.
Since the first two terms in the formula \eqref{HV} for $H_V$ are
independent of $\bmu$, it follows that
\[
G_{V'}\ge G_V\Longleftrightarrow H_{V'}\ge H_V.
\]

Note that $G_V(m) = \sum_{i=1}^d \max(0,m-n-\mu_i+1)$ since
$\dim_{\sfk}(R_\ell) = \max(0,\ell+1)$.  Let $\alpha_i = n+\mu_i-1$.
Then $\alpha_1 \le \cdots \le \alpha_d$ and $G_V(m) = \sum_{i=1}^d
\max(0,m-\alpha_i)$.

For simplicity, assume that $\alpha_1 < \cdots < \alpha_d$.  Then one
checks that
\[
G_V(\alpha_i) = (\alpha_i-\alpha_1) + \cdots +
(\alpha_{i}-\alpha_{i-1}) = (i-1)\alpha_i-\alpha_1 - \cdots
-\alpha_{i-1}.
\]
The graph of $G_V$ consists of the points $(\alpha_i, G_V(\alpha_i))$
linked by line segments of slopes $0,1,2,\dots,d$, where the segments
of slopes $0$ and $d$ are unbounded:
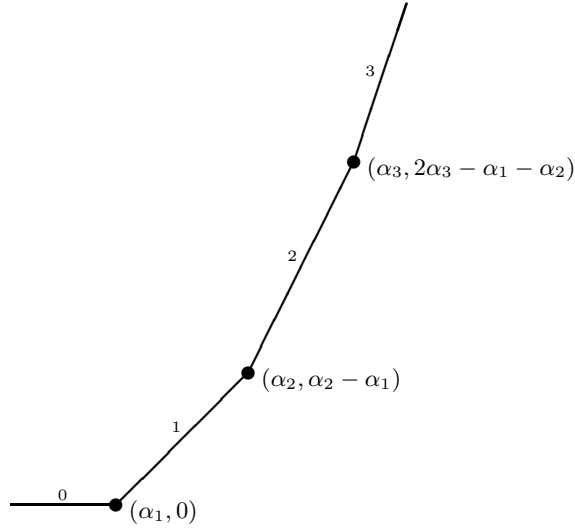
\[
\begin{picture}(233,215)
\thicklines
\put(55,15){$(\alpha_1,0)$}
\put(105,65){$(\alpha_2,\alpha_2-\alpha_1)$}
\put(145,145){$(\alpha_3,2\alpha_3-\alpha_1-\alpha_2)$}
\put(10,20){\line(1,0){40}}
\put(50,20){\circle*{5}}
\put(50,20){\line(1,1){50}}
\put(100,70){\circle*{5}}
\put(100,70){\line(1,2){40}}
\put(140,150){\circle*{5}}
\put(140,150){\line(1,3){20}}
\put(28,22){$\scriptstyle 0$}
\put(71,47.5){$\scriptstyle 1$}
\put(115,112.5){$\scriptstyle 2$}
\put(144.5,182.5){$\scriptstyle 3$}
\end{picture}
\]

\vskip-12pt

\noindent On the interval $\alpha_i \le x \le \alpha_{i+1}$, $G_V(x)$
is linear of slope $i$ and hence is given by
\[
G_V(x) = ix -\alpha_1 - \cdots -\alpha_{i}.
\]
Thus the region above the graph is defined by the inequalities
\begin{equation}
\label{graphineq}
y \ge 0,\ y \ge ix -\alpha_1 - \cdots -\alpha_{i},\ 1 \le i \le d.
\end{equation}

Now suppose $\bmu' = (\mu_1',\dots,\mu_d')$ comes from $V' \subseteq
R_n$ and set $\alpha_i' = n+\mu_i'-1$.  Then we have the following
equivalences:
\begin{align*}
G_{V'} \ge G_V &\iff \text{the graph of $G_{V'}$ lies
above the graph of $G_V$}\\
&\iff \text{the graph of $G_{V'}$ satisfies the inequalities
  \eqref{graphineq}}\\ 
&\iff (\alpha_i',(i-1)\alpha_i'-\alpha_1' - \cdots
-\alpha_{i-1}') \text{ satisfies  \eqref{graphineq} for all } i.
\end{align*}

A straightforward computation shows that
$(\alpha_i',(i-1)\alpha_i'-\alpha_1' - \cdots -\alpha_{i-1}')$
satisfies $y \ge ix -\alpha_1 - \cdots -\alpha_{i}$ if and only if
\[
\alpha_1'+ \cdots + \alpha_i'\le \alpha_1+\cdots +\alpha_i.
\]
This inequality holds for all $i$ when $G_{V'} \ge
G_V$.  Then $\bmu' \le \bmu$ follows immediately since $\alpha_i' =
n+\mu_i'-1$ and $\alpha_i = n+\mu_i-1$.  The converse takes more
work, since one has to prove that $\bmu' \le \bmu$ implies that for
all $i,j$, $(\alpha_i',(i-1)\alpha_i'-\alpha_1' - \cdots
-\alpha_{i-1}')$ satisfies $y \ge jx -\alpha_1 - \cdots -\alpha_{j}$.
We omit the details.

Finally, part (3) of the lemma follows immediately from part (2).
\end{proof}

The graph in the above proof is related to a
Harder-Narasimham partial order on the direct sums of line bundles
on $\PP^1$ (see  \cite[Definition 2.26]{I2}).

We complete the dictionary between Hilbert functions and $\bmu$-types
as follows. 

\begin{lemma}
\label{HVmu}
Fix $n$ and $d$ with $n \ge d$.  Then the map sending $H_V$ to the
$\bmu$-type $\bmu$ of the ideal $(V)$ induces a
well-defined bijection between the following two sets: 
\begin{enumerate}
\item The set of Hilbert functions $T$ such that $T = H_V$ for some
  subspace $V \in \Grass(d+1,R_n)$. 
\item The set of $d$-part partitions $\bmu$ satisfying $d \le |\bmu|
  \le n$.  
\end{enumerate}
It follows that there are only finite many Hilbert functions $T$ in
{\rm (1)}.
\end{lemma}

\begin{remark}
We use $T$ to denote a Hilbert function of the form $H_V$ for $V \in
\Grass(d+1,R_n)$ in order to match the notation of \cite{I2}.  The
reason for using $T$ will become clear later in the appendix.
\end{remark}

\begin{proof}
Lemma~\ref{mucomparisonlem}(3) implies that $H_V \mapsto
\bmu$ gives a well-defined injection from (1) to (2).  It remains to
prove that it is a surjection, i.e., that every $d$-part partition
from (2) is the $\bmu$-type of an ideal $(V)$ for some $V
\in \Grass(d+1,R_n)$. 

Given $\bmu = (\mu_1,\dots,\mu_d)$ as in (2), define $T : \mathbb{N}
\to \mathbb{N}$ by  
\[
T(m) = \dim_{\sfk}(R_m) - (d+1)\dim_{\sfk}(R_{m-n}) + \sum_{i=1}^d
\dim_{\sfk}(R_{m-n-\mu_i}). 
\]
Using $\dim_{\sfk}(R_\ell) = \max(0,\ell+1)$, one easily check that $H$
satisfies
\[
T(m) = \begin{cases} m+1, & \mbox{if}\ 0 \le m \le n-1,\\ n-d, &
  \mbox{if}\ m = n,\\  
n-|\bmu|, & \mbox{if}\ m \gg 0.\end{cases}
\]
Furthermore, the inequality $\max(0,\ell-1) + \max(0,\ell+1) \ge
2\hskip.5pt \max(0,\ell)$ makes it easy to show that
\[
T(m-1) + T(m+1) \ge 2T(m) \text{ whenever } m \ge n.
\]
Setting $e(m) = T(m-1)-T(m)$, it follows that $e(m) \ge e(m+1)$ for
all $m \ge n$.  By \cite[Proposition 4.6]{I1}, we conclude that $T =
H_V$ for some $V \in \Grass(d+1,R_n)$.  This proves the desired
surjectivity.
\end{proof}

Given a Hilbert function $T$ as in Lemma~\ref{HVmu}(1), we let
$\mathrm{GA}_T(d+1,n)$ be the set of all $V \in \Grass(d+1,R_n)$ such
that $T$ is the Hilbert function of $R/(V)$ (see \cite[Definition
2.16]{I2}).  Since there are only finitely many $T$'s, the
$\mathrm{GA}_T(d+1,n)$'s partition $\Grass(d+1,R_n)$ into finitely
many disjoint sets.

From \cite[Theorems 2.17 and 2.32]{I2} we have

\begin{theorem}\label{key78thm}  
Let $T$ be a Hilbert function as in {\rm Lemma~\ref{HVmu}(1)}.  Then:
\begin{enumerate}
\item $\mathrm{GA}_T(d+1,n)$ is irreducible.
\item The Zariski closure
  $\overline{\mathrm{GA}_T}(d+1,n)=\bigcup_{T'\ge T}
  \mathrm{GA}_{T'}(d+1,n)$, where the union is over all $T'\ge T$ from
  {\rm Lemma~\ref{HVmu}(1)}.
\end{enumerate}
\end{theorem}

We also have the following codimension result from \cite{I2}.
 
\begin{theorem}
\label{codthm}
Let $T$ be a Hilbert function as in {\rm Lemma~\ref{HVmu}(1)}, and let
$\bmu$ be the corresponding $d$-part partition. Then the codimension
of $\mathrm{GA}_T(d+1,n)$ in $\Grass(d+1,R_n)$ is given by the formula
\[
\cod(\mathrm{GA}_T(d+1,n)) = (n-|\bmu|)d+\sum_{i>j} \max(0, \mu_i-\mu_j-1).
\]
\end{theorem}

This follows from \cite[Theorem 2.24 (2.59)]{I2} since $n-|\bmu| =
\lim_{m\to\infty} T(m)$ and the partition $D$ from \cite[Definition
2.21]{I2} is just $\bmu$ written in descending order.

\begin{proof}[Proof of Theorems \ref{dimPndthm}, \ref{relprimeclosure}, 
\ref{gendim} and \ref{genclosure}.]  First note that Theorems
\ref{dimPndthm} and \ref{relprimeclosure} follow from Theorems
\ref{gendim} and \ref{genclosure} by intersecting with the open set
$\calP_{n,d} \subseteq \calCP_{n,d}$.

The next observation is that if $T$ corresponds to $\bmu$ via
Lemma~\ref{HVmu}, then $\Grass^\sbmu_{n,d}$ from Lemma~\ref{coverlem}
is precisely the set $\mathrm{GA}_T(d+1,n)$ since $R/(V)$ has Hilbert
function $T = H_V$ if and only if $(V)$ has $\bmu$-type $\bmu$.

Theorem~\ref{genclosure} is now an immediate consequence of Theorem
\ref{key78thm} via Lemmas~\ref{coverlem}, \ref{mucomparisonlem} and
\ref{HVmu}.  The irreducibility assertion of Theorem~\ref{gendim}
follows from Theorem~\ref{key78thm} and Lemma~\ref{coverlem}, and the
same results imply that
\[
\dim(\calP^\sbmu_{n,d}) = (d+1)(n+1) - \Big((n-|\bmu|)d+\sum_{i>j}
\max(0, \mu_i-\mu_j-1)\Big) 
\]
since $\dim(\calP_{n,d}) = \dim(R_n^{d+1}) = (d+1)(n+1)$.  This
easily reduces to the formula given in Theorem~\ref{gendim}.

It remains to show that $\calP^\sbmu_{n,d}$ is open in its Zariski
closure.  This follows from the disjoint union
\[
\overline{\calP^\sbmu_{n,d}} =  \calP^\sbmu_{n,d} \cup \bigcup_{\sbmu'
  < \sbmu} \calP_{n,d}^{\sbmu'} 
\]
since the large union on the right is easily seen to be closed by
Theorem~\ref{genclosure}.   
\end{proof}

\subsection{Proofs for Section~\ref{nonproper}}\label{nonproperproofs}

We begin with Proposition~\ref{nonproperprop1}.

\begin{proof}[Proof of Proposition~\ref{nonproperprop1}]
Half of the proof was given in Section~\ref{nonproper}.  For the other
half, assume $n \mid k$ and $n \ge kd$.  Then $n-kd+k \ge k$, so that
\[
\bmu = (\underbrace{k,\dots,k}_{d-1},n-kd+k)
\]
is a $d$-part partition of $n$.  Since $\bmu$ is divisible by $k$,
Theorem~\ref{nonproperthm} implies that $\calP_{n,d}^\sbmu$ contains
parametrizations of generic degree $k$. Thus the same is true for
$\calP_{n,d}$.
\end{proof}

We next turn to Proposition~\ref{nonproperprop}.

\begin{proof}[Proof of Proposition~\ref{nonproperprop}] 
First assume $\calP_{n,d}^\sbmu$ contains a parametrization
$(a_0,\dots,a_d)$ of generic degree $k$.  By L\"uroth's Theorem (see
Section 6.1 of \cite{SWP}), $n = km$, $m \in \Z$, and there are
relatively prime $\alpha,\beta \in R_k$ and $b_0,\dots,b_d \in R_m$
such that
\begin{equation}
\label{abab}
a_i(s,t) = b_i(\alpha(s,t),\beta(s,t)),\quad i = 0,\dots,d.
\end{equation}
(\cite{SWP} focuses on the affine case, but their treatment of
non-proper parametrizations easily translates to the projective
setting used here.)

Note that the $b_i$ are linearly independent and relatively prime
since the $a_i$ are (by assumption). Let $\tilde\bmu = (\tilde\mu_1, \dots,
\tilde\mu_d)$ be the $\bmu$-type of $(b_0,\dots,b_n)$.  Thus
$\tilde\mu_1 \le \cdots \le \tilde\mu_d$ and $\tilde\mu_1 + \cdots
+\tilde\mu_d = m$ since the $b_i$ are relatively prime.  

Substituting $\alpha,\beta$ into a $\bmu$-basis of
$(b_0,\dots,b_n)$ gives syzygies of $(a_0,\dots,a_n)$ of degrees
$k\tilde\bmu = (k\tilde\mu_1,\dots,k\tilde\mu_d)$.  We call these
\emph{composed syzygies}. If we can prove that the composed syzygies
form a $\bmu$-basis of $(a_0,\dots,a_n)$, then we will get the desired
result, namely $\bmu = k\tilde\bmu$.

Let $A'$ be the $(d+1)\times d$ matrix formed by the composed syzygies.
Since the maximal minors of a $\bmu$-basis of $(b_0,\dots,b_n)$ give the
$b_i$ up to sign, it follows that the maximal minors of $A'$ give the
$a_i$ up to sign.  Now let $A$ be the $(d+1)\times d$ matrix formed by
a $\bmu$-basis of $(a_0,\dots,a_n)$.  Its maximal minors also give the
$a_i$ up to sign.  Expressing each composed syzygy in terms of the
$\bmu$-basis gives a matrix equation
\[
A' = AQ
\]
where $Q$ is a $d\times d$ matrix of homogeneous polynomials.  Taking
maximal minors gives $a_i = a_i \det(Q)$ for all $i$, so that
$\det(Q) = 1$.  Hence $Q$ is an invertible matrix of scalars, which
proves that the composed syzygies are a $\bmu$-basis of
$(a_0,\dots,a_n)$, hence $\bmu$ is divisible by $k$.

To complete the proof, we next assume that $\bmu$ is divisible by $k$.
Then $\calP_{n,d}^\sbmu$ contains a parametrization of generic degree
$k$ by Theorem~\ref{nonproperthm}.
\end{proof}

The proof of Theorem~\ref{nonproperthm} will require more work.  

\begin{proof}[Proof of Theorem~\ref{nonproperthm}] 
First, note that $m\ge d$ is needed for the non-proper locus to be
non-empty.  Since $k$ divides $\bmu$ implies $n =
k\tilde\mu_1+...+k\tilde\mu_d$ and each $ \tilde\mu_i \ge 1$, this
implies $n \ge kd$ so that $m = n/k \ge d$.

Fix $d \ge 2$.  We will prove the theorem for all $k > 1$ and $n \ge
d$ by complete induction on $n$.  The base case $n=d$ is vacuously
true since $n = d$ implies $\bmu = (1,\dots,1)$, which is divisible by
no $k > 1$.

Now assume $n > d$ and that the theorem is true for all $m$ with $d
\le m < n$.  Take $\calP_{n,d}^\sbmu$ where $\bmu$ is a multiple of
$k$ and write $\bmu = k\tilde\bmu$, $\tilde\bmu =
(\tilde\mu_1,\dots,\tilde\mu_d)$.  As noted in the discussion leading
up to the theorem, this implies $k \mid n$.  Set $m = n/k$ and note
that $m = \tilde\mu_1+\cdots+\tilde\mu_d \ge d$.  Let
\[
U = \{(b_0,\dots,b_d) \in \calP_{m,d}^{\tilde\sbmu} \mid
(b_0,\dots,b_d) \text{ is proper}\}.
\]
Since $m < n$, our inductive hypothesis implies that $U$ is nonempty,
constructible, and Zariski dense in $\calP_{m,d}^{\tilde\sbmu}$.  In
particular, $\dim(U) = \dim(\calP_{m,d}^{\tilde\sbmu})$.

Let $W = \{(\alpha,\beta) \in R_k\times R_k \mid \alpha,\beta
\text{ are relatively prime}\}$.  Composing $(b_0,\dots,b_d) \in
U$ with $(\alpha,\beta) \in W$ gives $(a_0,\dots,a_d)$ as in
\eqref{abab}.  The $a_i$ have degree $n = km$ and are relatively prime
and linearly indepdendent since the $b_i$ are.  Furthermore, the
argument following \eqref{abab} shows that $(a_0,\dots,a_d)$ has
$\bmu$-type $\bmu = k\tilde\bmu$.  It follows that composition gives a
map
\begin{equation}
\label{nonproper:UWmap}
U \times W \longrightarrow \calP_{n,d}^{\sbmu},
\end{equation}
and the proof of Proposition~\ref{nonproperprop} shows that the image
of this map consists of all generic degree $k$ parametrizations in
$\calP_{n,d}^{\sbmu}$.  It follows easily that this locus is nonempty,
constructible, and has irreducible Zariski closure.  

To determine the codimension, we need to study the nonempty fibers of
\eqref{nonproper:UWmap}.  If a parametrization $(a_0,\dots,a_d)$ has
generic degree $k$ and image curve $C$, then the function field
$\sfk(C)$ can be identified with $\sfk(\alpha/\beta)$ for some
$(\alpha,\beta) \in V$ (in \cite{SWP},
$\alpha/\beta$ is denoted $R(t)$).  Since $\sfk(\alpha/\beta) =
\sfk(\alpha'/\beta')$ if and only if $\alpha/\beta$ and
$\alpha'/\beta'$ are related by a linear fractional transform, we see
that $(\alpha,\beta)$ is unique up to the action of
$\mathrm{GL}(2,\sfk)$.  Since this group has dimension $4$, it follows
that the nonempty fibers of \eqref{nonproper:UWmap} all have dimension
4.  Hence the generic degree $k$ locus in $\calP_{n,d}^{\sbmu}$ has
dimension
\[
\dim(U) + \dim(V) - 4 = \dim(\calP_{m,d}^{\tilde\sbmu}) + 2(k+1) - 4 = 
\dim(\calP_{m,d}^{\tilde\sbmu}) + 2(k-1).
\]
Hence the codimension is
\begin{equation}
\label{nonproper:codim1}
\dim(\calP_{n,d}^{\sbmu}) - \dim(\calP_{m,d}^{\tilde\sbmu}) - 2(k-1).
\end{equation}
Recall from Theorem~\ref{relprimedim} that
\begin{equation}
\label{nonproper:dimformulas}
\begin{aligned}
\dim(\calP_{n,d}^{\sbmu}) &= (d+1)(n+1) - \sum_{i > j}
\max(0,k\tilde\mu_i - k\tilde\mu_j
-1).\\ \dim(\calP_{m,d}^{\tilde\sbmu}) &= (d+1)(m+1) - \sum_{i > j}
\max(0,\tilde\mu_i - \tilde\mu_j -1).
\end{aligned}
\end{equation}
The following lemma will help compute the difference
$\dim(\calP_{n,d}^{\sbmu}) - \dim(\calP_{m,d}^{\tilde\sbmu})$.

\begin{lemma} 
\label{nonproper:Slem}
Given $\tilde\bmu = (\tilde\mu_1,\dots,\tilde\mu_d)$, let
$S(\tilde\bmu) = \sum_{i>j} \max(0,\tilde\mu_i-\tilde\mu_j)$.  Then:
\begin{enumerate}
\item $S(\tilde\bmu) = \sum_{i>j,\tilde\mu_i > \tilde\mu_j}
  \tilde\mu_i-\tilde\mu_j$.  
\item $\sum_{i>j} \max(0,\tilde\mu_i-\tilde\mu_j-1) = S(\tilde\bmu) -
  C$, where $C =\# \{(i,j) \mid i >j,\ \tilde\mu_i > \tilde\mu_j\}$.
\item $S(\tilde\bmu) \le (m-d)(d-1)$.
\end{enumerate}
\end{lemma}

\begin{proof}
The proof of (1) is straightforward, and for (2), we similarly get the
formula
\[ 
\sum_{i>j} \max(0,\tilde\mu_i-\tilde\mu_j-1) = \!\!\sum_{i>j,\tilde\mu_i
  > \tilde\mu_j}\! \tilde\mu_i-\tilde\mu_j-1. 
\]
From here, (2) follows easily.

We turn to (3).  If $m = d$, then the desired inequality is true since
the only possible $\tilde\bmu$ is $(1,\dots,1)$, for which $S = 0$.
Hence we may assume $m > d$.  Now write $\tilde\bmu =
(1,\dots,1,\tilde\mu_{i_0},\dots,\tilde\mu_d)$, where $\tilde\mu_{i_0}
> 1$.  Let $\tilde\bmu' = (1,\dots,\tilde\mu_{i_0}-1, \dots,
\tilde\mu_d+1)$.  If we can show that $S(\tilde\bmu) \le
S(\tilde\bmu')$, then it will follow that
\[
S(\tilde\bmu) \le S(1,\dots,1,m-d+1) = \sum_{d > j} (m-d+1) - 1 =
(m-d)(d-1),
\]
and the lemma will be proved.

When we compare $S(\tilde\bmu)$ and $S(\tilde\bmu')$, we only need to
consider pairs $i > j$ where $i = d$ or $i = i_0$ or $j = i_0$ (note
$i > j$ implies $j \ne d$).  We analyze these as follows:
\begin{itemize}
\item For terms with $i = d$, we have increased $\tilde\mu_d$ by $1$.
  Since $\tilde\mu_d+1$ is guaranteed to be bigger than every other
  entry, this increases $S\tilde(\bmu')$ by $d-1$.
\item For terms with $j = i_0$, we have decreased $\tilde\mu_{i_0}$ by
  $1$ and since we are subtracting, these terms increase
  $S(\tilde\bmu')$.
\item For terms with $i = i_0$, the possible $j$'s are
  $1,\dots,i_0-1$, and since we have decreased $\tilde\mu_{i_0}$ by
  $1$, we decrease $S(\tilde\bmu')$ by $i_0-1$.
\end{itemize}
Since $i_0 \le d$, the increase offsets the decrease, and
$S(\tilde\bmu) \le S(\tilde\bmu')$ follows.
\end{proof}
\noindent
{\it Completion of Proof of Theorem 4.5.}
By \eqref{nonproper:dimformulas} and Lemma~\ref{nonproper:Slem}, it
follows that
\begin{align*}
\dim(\calP_{n,d}^{\sbmu}) &= (d+1)(km+1) - (kS-C)\\
  \dim(\calP_{m,d}^{\tilde\sbmu}) &= (d+1)(m+1) - (S-C)
\end{align*}
since $n = km$ and $S=\sum_{i>j} \max(0,\tilde\mu_i-\tilde\mu_j) =
\sum_{i>j,\tilde\mu_i > \tilde\mu_j}
\tilde\mu_i-\tilde\mu_j$. Combining this with
\eqref{nonproper:codim1}, we see that the codimension is
\begin{align*}
&\dim(\calP_{n,d}^{\sbmu})-\dim(\calP_{m,d}^{\tilde\sbmu}) -2(k-1)\\
=\ &(d+1)(km+1) - (kS-C) -\big( (d+1)(m+1) - (S-C)\big)-2(k-1)\\
=\ &(k-1)(m(d+1) - S-2).
\end{align*}
This proves the desired formula for the codimension.

Since $S \le (m-d)(d-1)$ by Lemma~\ref{nonproper:Slem}, it follows
that 
\[
m(d+1) - S -2 \ge m(d+1) - (m-d)(d-1) -2 = d(d-1) +2m-2.  
\]
This easily gives the lower bound $(k-1)(d(d-1) +2m-2)$ stated in (1)
of the theorem.  Furthermore, $d(d-1) +2m-2 > 0$ since $m \ge d \ge 2$.
It follows that the codimension is always positive, completing the
proof of (1).

For (2), we consider all $k > 1$ that divide $\bmu$.  For any such
$k$, the locus of generic degree $k$ parametrizations has positive
codimension.  Since there are only finitely many such $k$'s, the same
is true for the non-proper locus.  Hence $\calP_{n,d}^\sbmu$ contains
a nonempty Zariski open subset consisting of proper parametrizations.
This subset is dense since $\calP_{n,d}^\sbmu$ is irreducible, and (2) follows.

When $d=1$, we have $S=C=0$ and the codimension formula readily
follows from \eqref{nonproper:codim1} and
\eqref{nonproper:dimformulas}.
\end{proof}

\subsection{Proofs for Section~\ref{smallest}}
\label{smallestproofs}
Ancestor ideals, not mentioned so far in this appendix, play a central
role in \cite{I2}.  As in Section~\ref{smallest}, an ideal $I
\subseteq R$ generated by elements of degree $n$ has an \emph{ancestor
ideal}, which is the largest homogeneous ideal of $R$ that agrees
with $I$ in degrees $m \ge n$.  When $I = (V)$ for $V \in
\Grass(d+1,R_n)$, we will follow \cite{I2} and denote its ancestor
ideal by $\overline{V}$.

We denote the Hilbert function of $R/\overline{V}$ by
$H_{\overline{V}}$.  These Hilbert functions are characterized in
\cite[Theorem 2.19]{I2}.  Given such a function $H$, let
$\Grass_H(d+1,n)$ consist of all $V \in \Grass(d+1,R_n)$ such that $H
= H_{\overline{V}}$ (see \cite[Definition 1.11]{I2}).

We define the partial order $\ge_\calP$ on these Hilbert functions by
setting $H' \ge_\calP H$ if and if $H'(m) \ge H(m)$ for $m \ge n$ and
$H'(m) \le H(m)$ for $0 \le m \le n$.  Then we have \cite[Theorem
  2.32]{I2}:

\begin{theorem}
\label{closureAmuthm}
$\Grass_H(d+1,n)$ is irreducible and open dense in its Zariski closure,
which is given by 
\[
\overline{\Grass_H}(d+1,n)=\bigcup_{H' \ge_\calP H}\Grass_{H'}(d+1,n).
\]
\end{theorem}

For the rest of the appendix, we will work in the relatively prime case.
Thus all partitions $\bmu = (\mu_1,\dots,\mu_d)$ that appear will be
partitions of $n$, i.e., $|\bmu| = n$.

Given $V \in \Grass(d+1,R_n)$, let $h_1, \ldots , h_\tau$ be minimal
generators of the ancestor ideal $\overline{V}$.  We assume $\deg(h_1)
\ge \cdots \ge \deg(h_\tau)$.  Note also that $\deg(h_i) \le n$ for
all $i$ since $(V)$ and $\overline{V}$ are equal in degrees $\ge n$.
Then \cite[(2.43)]{I2} implies the following.

\begin{lemma}
\label{ancprop}
Suppose $V \in \Grass(d+1,n)$ has $\bmu$-type $\bmu =
(\mu_1,\dots,\mu_d)$ with $|\bmu| = n$ and ancestor ideal
$\overline{V} = \langle h_1, \ldots , h_\tau\rangle$ as above.  Then
$\tau = d+1 -\#\{i \mid \mu_i = 1\}$ and we have a minimal free
resolution
\[
0 \longrightarrow {\textstyle\bigoplus_{i=1}^d} R(-n-\mu_i)
\longrightarrow {\textstyle\bigoplus_{i=1}^\tau} R(-\deg(h_i))
\longrightarrow R \longrightarrow R/\overline{V} \longrightarrow 0.
\]
\end{lemma}

This proposition implies in particular that
\begin{equation}
\label{Vdecomp}
V = \bigoplus_{i=1}^\tau R_{n-\deg(h_i)}\cdot h_i,
\end{equation}
so that if we set $\alpha_i = n+1-\deg(h_i)$, then
$\alpha_1+\cdots+\alpha_\tau = d+1$ and $\alpha_1 \le \cdots \le
\alpha_\tau$ since $\deg(h_1) \ge \cdots \ge \deg(h_\tau)$.  Thus
$\mathcal{A} = (\alpha_1,\dots,\alpha_\tau)$ is a $\tau$-part
partition of $d+1$.  Note also that $\bmu$ determines the length of
$\mathcal{A}$ since $\tau = d+1-\#\{i \mid \mu_i = 1\}$ by
Lemma~\ref{ancprop}.

It follows that $V$ gives two partitions, $\bmu$ and $\mathcal{A}$.
These partitions have a strong relation to the Hilbert function of the
ancestor ideal as follows.

\begin{lemma}
\label{twopartition}
Suppose $V \in \Grass(d+1,R_n)$ has partitions $\bmu$ and
$\mathcal{A}$, with $|\bmu| = n$, and Hilbert function $H =
H_{\overline{V}}$ of $R/\overline{V}$.  Given another $V' \in
\Grass(d+1,R_n)$ with partitions $\bmu'$ and $\mathcal{A}'$, such that $|\bmu'| = n$,
and $H' = H_{\overline{V}'}$, then
\[
H' \ge_\calP H \iff \bmu' \le \bmu \text{ and } \mathcal{A}' \le
\mathcal{A}.
\]
\end{lemma}

\begin{proof}
First suppose $H' \ge_\calP H$.  Since the ideals $(V)$ and
$\overline{V}$ are equal in degrees $\ge n$, it follows that $H_V(m) =
H_{\overline{V}}(m)$ for $m \ge n$.  The same is true for $V'$.  Since
$H_V(m) = H_{V'}(m) = m+1$ for $0 \le m \le n-1$, our assumption $H'
\ge_\calP H$ implies that $H_{V'} \ge H_V$, and then $\bmu' \le \bmu$
follows from Lemma~\ref{mucomparisonlem}.

From Lemma~\ref{ancprop}, we see that for $m \le n$,
\begin{align*}
H(m) &= \dim_{\sfk}(R_m) - \sum_{i=1}^\tau
\dim_{\sfk}(R_{m-\deg(h_i)})\\[-5pt] 
&= m+1 - \sum_{i=1}^\tau \max(0,m-\deg(h_i)+1).
\end{align*}
We write this as $H(m) = m+1 - G(m)$, where $G(m) = \sum_{i=1}^\tau
\max(0,m-\deg(h_i)+1)$, and similarly $H'(m) = m+1 - G'(m)$, where
$G'(m) = \sum_{i=1}^\tau \max(0,m-\deg(h_i')+1)$.

Then $H' \ge_\calP H$ implies $H'(m) \le H(m)$ for $m \le n$, so that
$G'(m) \ge G(m)$ for the same $m$.  The proof of
Lemma~\ref{mucomparisonlem} then implies that
\begin{equation}
\label{deghineq}
(\deg(h_\tau'),\dots,\deg(h_1')) \le (\deg(h_\tau),\dots,\deg(h_1))
\end{equation}
since $\deg(h_\tau) \le \cdots \le \deg(h_1)$, similarly for
$\deg(h_i')$.  Using $\alpha_i = n+1-\deg(h_i)$ and $\sum_{i=1}^\tau
\alpha_i= d+1$, one sees that
\[
\deg(h_\tau) + \cdots + \deg(h_{j+1})
=\alpha_1+\cdots+\alpha_j-(d+1)+(\tau-j)(n+1).
\]
It follows that \eqref{deghineq} is equivalent to $\mathcal{A}' =
(\alpha_1',\dots,\alpha_\tau') \le \mathcal{A} =
(\alpha_1,\dots,\alpha_\tau)$.

We omit the proof of the other implication.  
\end{proof}

Since the Hilbert function $H_V$ of $R/(V)$ equals the Hilbert
function $H_{\overline{V}}$ of $R/\overline{V}$ for $m \ge n$, we call
$T = H_V$ the \emph{tail} of $H = H_{\overline{V}}$ (see
\cite[Definition 2.16]{I2}).  This explains why we used $T$ for the
Hilbert functions occurring earlier in the appendix.

The final result we need is a consequence of \cite[Theorem 2.24]{I2}.

\begin{theorem}
\label{maincodthm}
Let $H$ be associated to partitions $\bmu = (\mu_1,\dots,\mu_d)$ and
$\mathcal{A} = (a_1,\dots,a_\tau)$ as above, and assume $|\bmu| = n$.
Then the codimension of $\Grass_H(d+1,n))$ in $\Grass(d+1,n)$ satisfies
\[
\cod(\Grass_H(d+1,n))= \sum_{i>j} \max(0,\mu_i - \mu_j - 1) +
\sum_{i>j} \max(0,a_i - a_j - 1).
\]
\end{theorem}
\begin{proof} Recall that the 
the partition $D$ of  \cite[Theorem 2.24]{I2} is our $\mu$ written in descending order, and $A$ there is our $\mathcal A$ written in descending order. The equation (2.61) of Theorem 2.24 there
can be rewritten when $c_H=0$ (no common factor of $(a_1,\ldots,a_n)$) as $\cod(\Grass_H(d+1,n))=\ell(A)+\ell(D)$, which is the expression above.
\end{proof}
\begin{proof}[Proofs of Theorems~\ref{ancestor1} and \ref{ancestor2}] 
Set $\bmu = \bmu_{\mathrm{min}} = (1,\dots,1,n-d+1)$ and note that
$\tau = (d+1) - (d-1) = 2$ since $n \ge d+1$.  Then Theorem
\ref{ancestor1} follows immediately from Lemma~\ref{ancprop} and
\eqref{Vdecomp}.

Next observe that the set $\calP^\sbmu_{n,d,\mathcal{A}}$ from Theorem
\ref{ancestor2} satisfies
\begin{equation}
\label{AGrassH}
\calP^\sbmu_{n,d,\mathcal{A}} = \pi^{-1}(\Grass_H(d+1,n)),
\end{equation}
where $\pi$ is from Lemma \ref{coverlem} and $H$ is the ancestor
Hilbert function corresponding to partitions $\bmu$ and $\mathcal{A}$.
Then the codimension formulas from Theorems~\ref{codthm} and
\ref{maincodthm}, together with Lemma \ref{coverlem}, show that
$\calP^\sbmu_{n,d,\mathcal{A}}$ has codimension $\max(0,a_2-a_1-1)$ in
$\calP^\sbmu_{n,d}$ (remember that $|\bmu| = n)$.

Finally, we compute the Zariski closure of
$\calP^\sbmu_{n,d,\mathcal{A}}$ in $\calP^\sbmu_{n,d}$.  Applying
$\pi^{-1}$ to Theorem \ref{closureAmuthm} and intersecting with
$\calP^\sbmu_{n,d}$, we obtain
\[
\overline{\mathcal P}^{\sbmu}_{n,d,\mathcal A}
=\bigg(\bigcup_{\sbmu' \le \sbmu,\mathcal A'\le \mathcal A}\mathcal
P^{\sbmu'}_{n,d,\mathcal A'}\bigg)\cap \calP^\sbmu_{n,d} 
\]
by \eqref{AGrassH} and Lemmas \ref{coverlem} and \ref{twopartition}.
Since $\calP^{\sbmu'}_{n,d,\mathcal{A}'} \subseteq
\calP^{\sbmu'}_{n,d}$ is disjoint from $\calP^\sbmu_{n,d}$ for $\bmu'
\ne \bmu$, the expression on the right reduces to the formula in
Theorem~\ref{ancestor2}.
\end{proof}

We note that $\calP^{\sbmu}_{n,d,\mathcal{A}} \subseteq
\calP^\sbmu_{n,d}$ can be defined for any $d$-part partition $\bmu$ of
$n$ and any $\tau$-part partition $\mathcal{A}$ of $d+1$, where $\tau
= d+1-\#\{i \mid \mu_i =1\}$ as in Lemma~\ref{ancprop}.
Theorems~\ref{ancestor1} and \ref{ancestor2} easily generalize to this
case using the above results from \cite{I2}.  Furthermore, if
$\mathcal{A} = (\alpha_1,\dots,\alpha_\tau)$ and $\tau \le d$, then
parametrizations in $\calP^{\sbmu}_{n,d,\mathcal{A}}$ give curves that
lie on the $\tau$-dimensional rational normal scroll
$S_{\alpha_1-1,\dots,\alpha_\tau-1} \subseteq \PP^d$.

\normalsize

\end{document}